\documentclass[12pt,leqno]{amsart}
\usepackage{amsmath,amssymb,amsfonts}
\usepackage[all]{xy}

\usepackage[T5,T1]{fontenc}

\usepackage{mathpazo}
\usepackage{hyperref}
\usepackage{a4wide}
\numberwithin{equation}{section}

\theoremstyle{plain}
\newtheorem{thm}{Theorem}[section]
\newtheorem{lem}[thm]{Lemma}
\newtheorem{cor}[thm]{Corollary}
\newtheorem{prop}[thm]{Proposition}

\theoremstyle{definition}
\newtheorem{defn}[thm]{Definition}
\newtheorem{ex}[thm]{Example}

\newtheorem{question}[thm]{Question}

\newtheorem{rmk}[thm]{Remark}
\newtheorem{rmks}[thm]{Remarks}

\newcommand{\surj}{\twoheadrightarrow}

\newcommand{\im}{{\rm im}}

\newcommand{\sC}{{\mathcal C}}

\newcommand{\sE}{{\mathcal E}}

\newcommand{\sL}{{\mathcal L}}
\newcommand{\sM}{{\mathcal M}}

\newcommand{\sR}{{\mathcal R}}
\newcommand{\sS}{{\mathcal S}}

\newcommand{\F}{{\mathbb F}}

\renewcommand{\H}{{\mathbb H}}

\newcommand{\Q}{{\mathbb Q}}

\newcommand{\U}{{\mathbb U}}

\newcommand{\Z}{{\mathbb Z}}

\def\NDT{{\fontencoding{T5}\selectfont Nguy\~ \ecircumflex n Duy T\^an}}
 
\begin{document}
\title{Counting Galois $\U_4(\F_p)$-extensions using Massey products} 

 \author{ J\'an Min\'a\v{c} and \NDT}
\address{Department of Mathematics, Western University, London, Ontario, Canada N6A 5B7}
\email{minac@uwo.ca}
 \address{Department of Mathematics, Western University, London, Ontario, Canada N6A 5B7 and Institute of Mathematics, Vietnam Academy of Science and Technology, 18 Hoang Quoc Viet, 10307, Hanoi - Vietnam } 
\email{duytan@math.ac.vn}
\thanks{JM is partially supported  by the Natural Sciences and Engineering Research Council of Canada (NSERC) grant R0370A01. NDT is partially supported  by the National Foundation for Science and Technology Development (NAFOSTED)}
 \begin{abstract}
We use Massey products and their relations to unipotent representations to parametrize and find an explicit formula for the number of Galois extensions of a given local field with the prescribed Galois group $\U_4(\F_p)$ consisting of unipotent four by four matrices over $\F_p$. 
Further applications of this method involve the counting of certain Galois extensions with restricted ramifications, and counting the numbers of Galois $\U_4(\F_p)$-extensions of some other fields. 
For each Demushkin pro-$p$-group, we find a very simple version of the condition when the $n$-fold Massey product of one-dimensional cohomological elements of $G$ with coefficients in $\F_p$, is defined.  As an easy consequence, we determine those $\U_n(\F_p)$ which occur as an epimorphic image of any given Demushkin group.
\end{abstract} 
\maketitle
\section{Introduction}
Recently, methods of Massey products arriving from several rather distinct entry points, have invaded the classical Galois theory of $p$-extensions. 
This development already resulted in new powerful techniques and results. In our paper we concentrate on the fundamental question of counting the number of certain Galois extensions over several significant base fields including
local fields and algebraic  number fields. Therefore our work is both influenced and contributes  further to the development discussed in the basic works of H. Koch, J. Labute and D. Vogel. (See \cite{La},\cite{Ko},\cite{Vo,Vo2}.)

In this introduction we merely point out some of these entry points, and explain how they have led us to a new way of counting the  number of specific Galois extensions. 
Massey products originated in topology, in an effort to produce finer topological invariants than those which existed before. (See \cite{Ma}.)
In the middle of the 1970s it was  recognized that the "non-vanishing of Massey products" could be viewed as an obstruction to the determination of the homotopy type of some topological spaces from their cohomological rings. (See \cite{DGMS}.) At the same time, a crucial link  between unipotent representations of groups, and a certain vanishing of the Massey product, was discovered. (See \cite{Dwy}.)  

In \cite{HW}, M. J. Hopkins and K. G.  Wickelgren proved that all triple Massey products with respect to $\F_2$, vanish over all global fields of characteristic not 2. In \cite{MT1} we proved that this is true for all fields $F$. We use this result to provide strong restrictions on the possible relations in Galois groups $G_F(2)$.
In \cite{MSp} some results of F. R. Villegas were used for the case $p=2$ to describe the quotient of the Galois group $G_F(2)$ of the quadratic closure of a field $F$ of ${\rm char}(F)\not=2$. These results were extended to all primes $p$ in \cite{EM1} and \cite{EM2} for both $p$-descending and Zassenhaus filtrations. In \cite{Ef} a beautiful, unifying description via unipotent representations in $\U_3(\F_p)$ of these results in \cite{MSp} and \cite{EM2} was obtained, and at the same time some important results on $n$-unipotent representations were obtained for all $n\geq 3$.  Influenced by these results and also by further considerations, 
we formulated a conjecture on the vanishing of $n$-fold Massey products and the kernel unipotent conjecture related to the maximal pro-$p$-quotients of absolute Galois groups of fields containing a primitive $p$th root of unity. These possible conditions on absolute Galois groups are rather strong. We verified the kernel unipotent conjecture for rigid fields when $p>2$ (see \cite{MTE}). The vanishing of 3-Massey product conjecture was recently verified in \cite{Mat}, \cite{EMa} \cite{MT3}, and \cite{MT4}.

In the work concerning  these two conjectures, it turned out that it was important to understand the Galois extensions of a given field with Galois group $\U_n(\F_p)$. (Recall that $\U_n(\F_p)$ is the group of all upper-triangular unipotent $n$-by-$n$ matrices with entries in the finite field $\F_p$.) Moreover we would like to parametrize and count these extensions via techniques of Massey products. The first important non-trivial case not yet solved in the literature, is the case $n=4$  and $F$ local fields. We solve this problem here. We further apply these techniques to other fields as well. We also found out that we can considerably relax the condition on defining  $n$-fold Massey products for any $n\geq 2$ and for any Demushkin group. The statement that this relaxed condition is sufficient to define $n$-fold Massey products seems rather strong, and we point out a relatively simple example when this condition is satisfied, and yet the $n$-fold Massey product is not defined. Nevertheless there is the possibility that this relaxed condition is always sufficient if our pro-$p$-group is the maximal pro-$p$-quotient  of an absolute Galois group of a field $F$ containing a primitive $p$th root of unity. 
This may lead to a new restriction of considerable strength, on the shape of of such profinite groups.
Now we shall turn to the details of our paper.

Let $K$ be a local field of characteristic $0$, i.e., $K$ is a finite extension of the field $\Q_p$ of $p$-adic numbers. It is well-known that $K$ has only  finitely many algebraic extensions (inside a fixed algebraic closure of $K$) with given degree. In particular, the number of Galois extensions of $K$ with prescribed finite Galois group $G$, is finite. We  follow \cite{Ya} to denote this number by $\nu(K,G)$. Finding  $\nu(K,G)$ for various finite $p$-groups is a classical problem. 
Assume that $G$ is a finite $p$-group. If $K$ does not contain a primitive $p$th root of unity, then Shafarevich \cite{Sha} proved that the Galois group $G_K(p)$ of the maximal $p$-extension of $K$, is a free pro-$p$-group of rank $n+1$, where $n=[K:\Q_p]$. He also provided an explicit formula for $\nu(K,G)$ in this case. If $K$ does contain a primitive $p$th root of unity, then the group $G_K(p)$ is a Demushkin group. Demushkin groups are completely classified by S.P. Demushkin, J.-P. Serre, and J. Labute (see \cite{La}). In \cite{Ya}, the author provided a general formula for $\nu(K,G)$ in this case using the classification of Demushkin groups, complex character theory of finite groups, and the M\"{o}bius inversion formula. Theoretically, if we know the character table of G and those of its subgroups, then we can determine $\nu(K,G)$. However in practice, it is not easy to do so. In \cite{Ya} the author applied this general formula for some special $p$-groups, those having quite simple character theory, e.g., groups of order $p^3$,  dihedral and  generalized quaternion groups of 2-power orders. In particular for the group $G=\U_3(\F_p)$, he recovered the result of R. Massy and T. Nguyen-Quang-Do \cite{MNg} when $p>2$, and of C. Jensen and N. Yui \cite{JY} when $p=2$. 

In this paper we shall provide an explicit formula for $\nu(K,\U_4(\F_p))$.  For example, $\Q_2$ has exactly 16 Galois extensions with the Galois group $\U_4(\F_2)$. It seems that it is technically difficult to find $\nu(K,\U_4(\F_p))$ using the above-mentioned group character theory approach. Here we use Massey products, more precisely triple Massey products, in computing $\nu(K,\U_4(\F_p))$. By the work \cite{Dwy}, triple Massey products  appear  naturally in this problem. Every surjective homomorphism from $G_K(p)$ to $\U_4(\F_p)$ determines a defined triple Massey product which in fact contains 0.  We can think of the set ${\rm Epi}(G_K(p),\U_4(\F_p))$ of surjective homomorphisms from $G_K(p)$ to $\U_4(\F_p)$ as a fiber space above the base space ${\rm TMP}(G_K(p),\F_p)$ consisting of defined triple Massey products containing 0. Then the problem of determining  ${\rm Epi}(G_K(p),\U_4(\F_p))$ is  reduced to two problems: determining the base space   ${\rm TMP}(G_K(p),\F_p)$ and determining the fibers. Using a general result in embedding problems, we can describe each fiber of this fiber space in  cohomological terms. (See  Section~\ref{sec:Massey} for more precise discussions.) The problem of determining ${\rm TMP}(G_K(p),\F_p)$ is in fact a problem in linear algebra (see Lemma~\ref{lem:LA}).
We would also like to note that using cup products instead of triple Massey products, the same method as above can be applied to calculate $|\nu(K,\U_3(\F_p))|$ (see Remark~\ref{rmk:CP}). In Subsection 2.4 we illustrate our method on counting the number of Galois extensions with restricted ramification.
 
The structure of our paper is as follows. In Section 2 we review Massey products and embedding problems. We provide a general parametrization for the set ${\rm Epi}(G,\F_p)$ in  Proposition~\ref{prop:counting Epi}. In Subsection 2.4  we  specialize this parametrization to obtain Proposition~ \ref{prop:R small}, when $G=S/R$ with $S$ free pro-$p$-group and $R\subseteq S_{(3)}$. 
Further  we develop a new  technique which we illustrate by  examples for counting the number of $\U_4(\F_2)$-Galois extensions of $\Q$ with restricted ramification. In Section 3 we prove the main results of this paper. These results are  Theorem~\ref{thm:counting U4 Demushkin} and Theorem~\ref{thm: local U4-extensions}. In Section 4 we discuss the question when ${\rm Epi}(G,\U_n(\F_p))$ is non-empty, where $G$ is  a Demushkin group. We answer this question as an easy corollary of Proposition~\ref{prop: Demushkin n-Massey}. The latter shows that we can relax the   defining condition for a $n$-fold Massey product for a Demushkin group.  We also show in Proposition~\ref{prop:free product} that  two pro-$p$-groups satisfy this relaxed condition on defining Massey products if  and only if their free product satisfies this condition.
In the last section, we find a formula for $|{\rm Epi}(G,\U_4(\F_p))|$ when $G$ is a free product of a Demushkin group and a free pro-$p$-group; or  $G$ is a free product of two Demushkin groups with odd $q$-invariants.
\\
\\
{\bf Acknowledgement: } We would like to thank Danny Neftin and Michael Rogelstad for their interest and encouragement related to this work. We are  grateful to an  anonymous referee for his/her  comments and valuable suggestions which we used to improve our exposition.
\\
\\
{\bf Notation:} For two profinite groups $G$ and $H$, we denote by ${\rm Epi}(G,H)$ the set consisting of continuous surjective homomorphisms from $G$ to $H$. (All homomorphisms of profinite groups considered in this paper are assumed to be continuous.) 

For a prime number $p$ and a field $F$, we denote by $G_F(p)$ the maximal pro-$p$-quotient of an absolute Galois group $G_F$ of $F$. If $E/F$ is a Galois extension with Galois group isomorphic to $G$, then we say that $E/F$ is a $G$-extension.

For a profinite group $G$ and a prime number $p$,  the Zassenhaus ($p$-)filtration $(G_{(n)})$ of $G$ is defined inductively by
\[
G_{(1)}=G, \quad G_{(n)}=G_{(\lceil n/p\rceil)}^p\prod_{i+j=n}[G_{(i)},G_{(j)}],
\]
where $\lceil n/p \rceil$ is the least integer which is greater than or equal to $n/p$. (Here for two closed subgroups $H$ and $K$ of $G$,  $[H,K]$ means the smallest closed subgroup of $G$ containing the  commutators $[x,y]=x^{-1}y^{-1}xy$, $x\in H, y\in K$. Similarly, $H^p$ means the smallest closed subgroup of $G$ containing  the $p$-th powers $x^p$, $x\in H$.)
\section{Massey products and embedding problems}
\label{sec:Massey}
\subsection{Massey products}
Let $G$ be a profinite group and $p$ a prime number. We consider the finite field $\F_p$ as  a trivial discrete $G$-module. Let $\sC^\bullet=(\sC^\bullet(G,\F_p),\delta,\cup)$ be the differential graded algebra of inhomogeneous continuous cochains of $G$ with coefficients in $\F_p$ \cite[Ch.\ I, \S2]{NSW}. We write $H^i(G,\F_p)$ for the corresponding cohomology groups. We denote by $Z^1(G,\F_p)$ the subgroup of $C^1(G,\F_p)$ consisting of all 1-cocycles. Because we use  trivial action on the coefficients $\F_p$, $Z^1(G,\F_p)=H^1(G,\F_p)={\rm Hom}(G,\F_p)$.
In this section we review  Massey products in $H^\bullet(G,\F_p)$ and their relations to certain types of embedding problems, which will be needed in the sequel.  (See \cite{MT1,MT2} and references therein for more general setups.)

Let $n\geq 3$ be an integer. Let $a_1,\ldots,a_n$ be elements in $H^1(G,\F_p)=Z^1(G,\F_p)\subseteq C^1(G,\F_p)$.
\begin{defn}
\label{defn:Massey product}
 A collection $\sM=\{a_{ij}\mid 1\leq i<j\leq n+1, (i,j)\not=(1,n+1)\}$ of elements $a_{ij}$ of $\sC^1(G,\F_p)$ is called a {\it defining system} for the {\it $n$-fold Massey product} $\langle a_1,\ldots,a_n\rangle$ if the following conditions are fulfilled:
\begin{enumerate}
\item $a_{i,i+1}=a_i$ for all $i=1,2\ldots,n$.
\item $\delta a_{ij}= \sum_{l=i+1}^{j-1} a_{il}\cup a_{lj}$ for  all $i+1<j$.
\end{enumerate}
Then $\sum_{k=2}^{n} a_{1k}\cup a_{k,n+1}$ is a $2$-cocycle.
Its  cohomology class in $H^2$  is called the {\it value} of the product relative to the defining system $\sM$,
and is denoted by $\langle a_1,\ldots,a_n\rangle_\sM$.
The product $\langle a_1,\ldots, a_n\rangle$ itself is the subset of $H^2(G,\F_p)$ consisting of all  elements which can be written in the form $\langle a_1,\ldots, a_n\rangle_\sM$ for some defining system $\sM$.

When $n=3$ we will speak about a {\it triple} Massey product. Note that in this case the triple Massey  product $\langle a_1,a_2,a_3\rangle$ is defined if and only if $a_1\cup a_2=a_2\cup a_3=0$ in $H^2(G,\F_p)$.
\end{defn}
\begin{rmk}
\label{rmk:Dwyer}
As observed by Dwyer \cite{Dwy} in the discrete context (see also \cite[\S8]{Ef} in the profinite case), defining systems for Massey products can be interpreted in terms of upper-triangular unipotent representations of $G$, as follows.

Let $\U_{n+1}(\F_p)$ be the group of all upper-triangular unipotent $(n+1)\times(n+1)$-matrices with entries in  $\F_p$. Let $Z$ be the subgroup of all such matrices with all off-diagonal entries being $0$ except at position $(1,n+1)$. We may identify $\U_{n+1}(\F_p)/Z$ with the group $\bar\U_{n+1}(\F_p)$ of all upper-triangular unipotent $(n+1)\times(n+1)$-matrices with entries over $\F_p$ with the $(1,n+1)$-entry omitted.
For any representation $\rho\colon G\to \U_{n+1}(\F_p)$ and $1\leq i< j\leq n+1$, let $\rho_{ij}\colon G\to \F_p$ be the composition of $\rho$ with the projection from $\U_{n+1}(\F_p)$ to its $(i,j)$-coordinate. We use  similar notation for representations $\bar\rho\colon G\to\bar\U_{n+1}(\F_p)$. Note that for each $i=1,\ldots,n$, $\rho_{i,i+1}$ (resp., $\bar\rho_{i,i+1}$) is a group homomorphism.

Assume that $\sM=\{a_{ij}\mid 1\leq i<j\leq n+1, (i,j)\not=(1,n+1)\}$ is a defining system for an $n$-fold Massey product $\langle a_1,\ldots,a_n\rangle$. We define a map $\bar{\rho}_{\sM}\colon G\to \bar{\U}_{n+1}(\F_p)$ by $(\bar\rho_\sM)_{ij}=-a_{ij}$.  Then one can check that $\bar{\rho}_\sM$ is a (continuous) group homomorphism. Moreover, $\langle a_1,\ldots,a_n\rangle_\sM=0$ if and only if $\bar{\rho}_\sM$ can be lifted to a group homomorphism $G\to \U_{n+1}(\F_p)$.
On the other hand, if  $\bar{\rho}\colon G \to \bar{\U}_{n+1}(\F_p)$ is a group homomorphism, then $\{-\bar{\rho}_{ij}\mid 1\leq i<j\leq n+1, (i,j)\not=(1,n+1)\}$ is a defining system for $\langle -\bar{\rho}_{12},\ldots,-\bar{\rho}_{n,n+1}\rangle$. (See \cite[Theorem 2.4]{Dwy}.) 
\qed
\end{rmk}

\begin{defn}
 We say that $G$ has the {\it vanishing triple Massey product property (with respect to $\F_p$)} if every defined triple Massey product $\langle a_1,a_2, a_3\rangle$, where $ a_1,a_2, a_3\in H^1(G,\F_p)$, necessarily contains $0$.
\end{defn}

\subsection{Embedding problems} A {\it weak embedding problem} $\sE$ for a profinite group $G$ is a diagram 
\[
\sE:=
\xymatrix
{
{} & G \ar[d]^{\varphi}\\
U \ar[r]^f & \bar U
}
\] 
which consists of  profinite groups $U$ and $\bar U$ and homomorphisms $\varphi \colon G\to \bar U$, $f\colon U\to \bar U$ with $f$ being surjective.  If in addition $\varphi$ is also surjective, we call $\sE$ an {\it embedding problem}. 

A {\it weak solution} of $\sE$ is  a homomorphism $\psi\colon G\to U$ such that $f\psi=\varphi$.  
 We call $\sE$ a {\it finite} weak embedding problem if  group $U$ is finite. The {\it kernel} of $\sE$ is defined to be $M:=\ker(f)$.  
We denote by ${\rm Sol}(\sE)$ the set of weak solutions of $\sE$. 

Assume now that the kernel $M$ is abelian. The conjugation action of $U$ on $M$ is trivial when we restrict this action to $M\subseteq U$. Hence this induces an $\bar{U}$-module structure on $M$. We consider $M$ as a $G$-module via $\varphi$ and the conjugation action of $\bar{U}$ on $M$. We denote by $M_\varphi$ this $G$-module.
The following result is well-known.

\begin{lem}
\label{lem:Sol}
Let $\sE(G,f,\varphi)$ be a weak embedding problem with finite abelian kernel $M$ which has a weak solution. Then ${\rm Sol}(\sE)$ is a principal homogeneous space over the group of 1-cocycles $Z^1(G,M_\varphi)$.

In particular, any weak solution $\theta$ of $\sE$ induces a bijection 
\[
{\rm Sol}(\sE) \simeq Z^1(G,M_\varphi).
\]
\end{lem}
\begin{proof} See \cite[Proof of 3.5.11]{NSW}.
\end{proof}
\begin{ex}
\label{ex:Z1 free}
  Let $\sE(G,f\colon U\to \bar{U},\varphi\colon G\to \bar{U})$ be a weak embedding problem with finite abelian kernel $M$. Assume that $G$ is a free pro-$p$-group on generators $x_1,\ldots,x_d$,  and $U$ is a finite $p$-group. Then $|Z^1(G,M_\varphi)|=|M|^d$. (Observe that this means that one-cocycles can be arbitrarily prescribed on generators.)   Indeed, $\sE$ has a weak solution, as $G$ is free. Hence  by Lemma~\ref{lem:Sol} we have $|Z^1(G,M_\varphi)|=|{\rm Sol}(\sE)|$. A homomorphism $\psi\colon G\to U$ is an element in ${\rm Sol}(\sE)$ if and only if 
\[
\psi(x_i)\bmod M =\varphi(x_i), \quad \forall 1\leq i\leq d.
\]
Hence the number of such $\psi$'s is $|M|^d$. Therefore, $|Z^1(G,M_\varphi)|=|{\rm Sol}(\sE)|=|M|^d$.
\end{ex}
\subsection{Epimorphisms to $\U_4(\F_p)$}
We first have the following lemma, which will be useful later.
\begin{lem}
\label{lem:linear independence} Let $G$ be a pro-$p$-group.
 Let $\chi_1,\ldots,\chi_n$ be elements in $H^1(G,\F_p)$. Then the homomorphism \[\varphi:=(\chi_1,\ldots,\chi_n)\colon G\to \F_p\times\cdots\times \F_p\] is surjective if and only if $\chi_1,\ldots,\chi_n$ are $\F_p$-linearly independent in $H^1(G,\F_p)$.
\end{lem}
\begin{proof}
We set $H:=\F_p\times\cdots\times \F_p$. Then $\varphi\colon G\to H$ is surjective if and only if the induced homomorphism $\varphi^*\colon H^1(H,\F_p)\to H^1(G,\F_p)$ is injective (\cite[Proposition 1.6.14 (ii)]{NSW}). We have a  (non-canonical) isomorphism
\[
H\to H^1(H,\F_p), a=(a_1,\ldots,a_n)\mapsto \chi_a,
\]
where $\chi_a$ is defined by $\chi_a(h_1,\ldots,h_n)=\sum_{i=1}^n a_ih_i$. Then for each $a=(a_1,\ldots,a_n)\in H$, 
\[
(\varphi^*(\chi_a))(g)=\sum_{i=1}^n a_i\chi_1(g),\; \forall g\in G.
\]
Therefore $\varphi^*$ is injective if and only if $\chi_1,\ldots,\chi_n$ are $\F_p$-linearly independent. 
\end{proof}
The following two lemmas provide some general properties of the set ${\rm Epi}(G,\U_n(\F_p))$.
\begin{lem}
\label{lem:G mod Zassenhaus subgroup} Let $G$ be a profinite group, and let $G(p)$ be its maximal pro-$p$-quotient. Then
\begin{enumerate}
\item[(a)] ${\rm Epi}(G,\U_n(\F_p))\simeq {\rm Epi}(G(p),\U_n(\F_p))$,
\item[(b)] ${\rm Epi}(G,\U_n(\F_p))\simeq {\rm Epi}(G/G_{(n)},\U_n(\F_p))$.
\end{enumerate}
\end{lem}
\begin{proof} (a) This is clear since $\U_n(\F_p)$ is a finite $p$-group.

(b) This follows from the fact that every homomorphism $\rho\colon G\to \U_n(\F_p)$ is trivial on $G_{(n)}$ (see for example \cite[Lemma 3.7]{MT1}, or \cite[Lemma 2.5]{MTE}, see also \cite[Corollary 7.2]{Ef}). 
\end{proof}

\begin{lem}
\label{lem:equiv}
Let $N,N^\prime$ be closed normal subgroups of a free pro-$p$-group $S$ such that $NS_{(n)}=N^\prime S_{(n)}$. Then 
\[
{\rm Epi}(S/N,\U_n(\F_p))\simeq {\rm Epi}(S/N^\prime,\U_n(\F_p)).
\]
\end{lem}
\begin{proof} Let $G:=S/N$ and $G^\prime:=S/N^\prime$. Because surjective homomorphisms take $n$th Zassenhaus filtrations onto $n$th Zassenhaus filtrations,  using our assumption, we have
\[
G/G_{(n)}\cong S/NS_{(n)}=S/N^\prime S_{(n)}\cong G^\prime/G^\prime_{(n)}.
\]
Therefore our result follows from Lemma~\ref{lem:G mod Zassenhaus subgroup}.
 \end{proof}

We now consider the following exact sequence of finite groups
\begin{equation}
\label{eq:U4 exact sequence}
1 \longrightarrow M \longrightarrow \U_4(\F_p) \xrightarrow{(a_{12},a_{23},a_{34})} (\F_p)^3\longrightarrow 1,
\end{equation}
where $a_{ij}\colon\U_4(\F_p)\to\F_p$ is the map sending a matrix to its $(i,j)$-coefficient.  We set $B:=(\F_p)^3$.  
Let $\varphi\colon G \to B$ be any continuous homomorphism. We consider $M$ as a $G$-module via $\varphi$ and the conjugation action of $B$ on $M$. We denote by $M_\varphi$ this $G$-module.

We define ${\rm TMP}(G,\F_p)$ to be the set of triples $(x,y,z)\in (H^1(G,\F_p))^3$ such that the triple Massey product $\langle x,y,z\rangle$ is defined and contains 0, and that $x,y,z$ are $\F_p$-linearly independent in $H^1(G,\F_p)$. For each $\varphi\in {\rm TMP}(G,\F_p)$, let us  still denote, by abuse 
of notation, $\varphi \colon G\to (\F_p)^3$ the induced surjective group homomorphism.

\begin{prop}
\label{prop:counting Epi}
Let the notation be as above. Assume that the set ${\rm TMP}(G,\F_p)$ is finite, and that for each $\varphi \in {\rm TMP}(G,\F_p)$, $|Z^1(G,M_{\varphi})|$ is finite.  
Then 
\[
|{\rm Epi}(G,\U_4(\F_p))|=\sum_{\varphi\in {\rm TMP}(G,\F_p)} |Z^1(G,M_\varphi)|. 
\]
 \end{prop}
 \begin{proof}
 This follows from Remark~\ref{rmk:Dwyer}, Lemma~\ref{lem:Sol} and Lemma~\ref{lem:linear independence}.
 \end{proof}

\begin{cor}
\label{cor:VTM}
 Let $G$ be a profinite group which has the vanishing triple Massey product property with respect to $\F_p$. Assume that for every surjective homomorphism $\varphi \colon G\surj (\F_p)^3$, $\dim Z^1(G,M_\varphi)=:s<\infty$, which is independent of $\varphi$. Assume further that  the number $N$ of triples $(a,b,c)\in (H^1(G,\F_p))^3$ such that $a\cup b=b\cup c=0$, and that $a,b$ and $c$ are $\F_p$ linearly independent, is finite. 
Then the number of surjective homomorphisms $G\surj \U_4(\F_p)$ is $Np^s$.
\end{cor}
\begin{proof} Since $G$ has the vanishing triple Massey product property with respect to $\F_p$, we see that the cardinality of ${\rm TMP}(G,\F_p)$ is equal to $N$. Note also that, by assumption, $|Z^1(G,M_{\varphi})|=p^s$ for every surjecrtive homomorphism $\varphi\colon G\surj (\F_p)^3$. The statement then follows immediately from  Proposition~\ref{prop:counting Epi}.
\end{proof}
\subsection{Numbers of $\U_4(\F_2)$-extensions with restricted ramification}

 We now recall the definition of trace maps. Let $G$ be a pro-$p$-group. 
Let
\[ 1 \to R \to  S \to  G \to 1,\]
be a minimal presentation of $G$, i.e., $S$ a free pro-$p$-group and $R\subset S^p[S,S]$. Then the inflation map
\[ \inf :H^1(G, \F_p) \to  H^1(S, \F_p)\]
is an isomorphism by which we identify both groups. Since $S$ is free, we have $H^2(S,\F_p) = 0$ and from the  5-term exact sequence we obtain that the transgression map 
\[ {\rm trg}\colon H^1(R,\F_p)^ G \to H^2(G,\F_p)\]
is an isomorphism. 
Therefore any element $r\in R$ gives rise to a map
\[ {\rm tr}_r\colon H^2(G,\F_p) \to \F_p,\]
which is defined by $\alpha\mapsto {\rm trg}^{-1}(\alpha)(r)$ and is called the {\it trace map} with respect to $r$.

 \begin{prop}
\label{prop:R small}
 Let $G=S/R$, where $S$ is a free pro-$p$-group of rank $n$, and $R$ is a normal subgroup of $S$ with $R\subseteq S_{(3)}$. Then
 \[
 |{\rm Epi}(G,\U_4(\F_p))|=|{\rm TMP}(G,\F_p)| p^{3n}
 \]
 \end{prop}
 \begin{proof} Let $\varphi=(\alpha,\beta,\gamma)$ be any fixed element in ${\rm TMP}(G,\F_p)$, which determines a surjective homomorphism, also  denoted by $\varphi\colon G\to(\F_p)^3$. It is enough to show that $|Z^1(G,M_\varphi)|=p^{3n}$. 
 
 To this end, let us first consider a minimal set $x_1,\ldots,x_n$ of topological generators of $S$ and its image $\bar{x}_1,\ldots,\bar{x}_n$ in $G$. Let $\tilde{\psi}$ be any homomorphism from $S$ to $\U_4(\F_p)$ such that 
 \[
 \tilde{\psi}(x_i) \bmod M=   -\varphi(\bar{x}_i), \quad \forall 1\leq i\leq n,
 \]
 or equivalently,
 \[
 \tilde{\psi}_{12}=-\alpha,  \tilde{\psi}_{23}=-\beta,  \tilde{\psi}_{34}=-\gamma.
 \]
 Note that $\tilde{\psi}(S_{(3)})=1$ in $\bar{\U}_4(\F_p)$. Hence $\tilde{\psi}$ induces a homomorphism $\bar{\psi}\colon G=S/R\to \bar{\U}_4(\F_p)$ such that the following diagram commutes
 \[\
\xymatrix{
S \ar@{->}[d]^{\tilde{\psi}} \ar@{->}[r] &G \ar@{->}[d]^{\bar{\psi}}\ar@{->}[r] &1\\
\U_4(\F_p)\ar[r] &\bar{\U}_4(\F_p)\ar[r] &1.
}
\]
Since $R\subseteq S_{(3)}$, it is well-known  (see \cite[Appendix]{Vo2}, see also \cite{Ef}) that we have a well-defined (single-valued) triple Massey product 
\[
\langle \cdot,\cdot,\cdot\rangle\colon H^1(G,\F_p)\times H^1(G,\F_p)\times H^1(G,\F_p) \to H^2(G,\F_p).
\]
Since $\varphi=(\alpha,\beta,\gamma)\in {\rm TMP}(G,\F_p)$, we obtain $\langle \alpha,\beta,\gamma\rangle =0$. Hence for every $r\in R$,  \cite[Lemma 3.4]{MT1} implies that
\[
0={\rm tr}_r(\langle\alpha,\beta,\gamma\rangle)=-\tilde{\psi}_{1,4}(r).
\]
Thus $\tilde{\psi}(r)=1$ for every $r\in R$. This implies that $\tilde{\psi}$ factors through a homomorphism $\psi\colon G\to \U_4(\F_p)$.
 Because $-\psi$ is a lift of $\varphi$ on topological generators of $G$, we see that $-\psi$ is a lift of $\varphi\colon G\to (\F_p)^3$.
  Therefore $|Z^1(G,M_\varphi)|$ is equal to the number of liftings of $\varphi$, which is in turn equal to $p^{3n}$, the number of such $\tilde{\psi}$'s. 
 \end{proof}

\begin{prop}
\label{prop:restricted ramification}
  Let $G=S/R$, where $S$ is a free pro-$p$-group on  $n$ generators $x_1,x_2,\ldots,x_n$, and $R$ is a normal (closed) subgroup of $S$ generated by $\rho_1,\ldots,\rho_r$. We assume that for each $1\leq m\leq r$ we have 
\[
\rho_m\equiv \prod_{1\leq i<j\leq n; k\leq j} [[x_i,x_j],x_k]^{e_{i,j,k,m}} \mod S_{(4)}. 
\]
Let $N$ be the number of  $(a_1,\ldots,a_n,b_1,\ldots,b_n,c_1,\ldots,c_n)\in \F_p^{3n}$ which satisfy the following two conditions:
\begin{enumerate}
\item For each $1\leq m\leq r$, we have
\[
\begin{split}
\sum_{i<j,k\leq j} (a_ib_jc_k-a_jb_ic_k+ a_kb_jc_i-a_kb_ic_j)e_{i,j,k,m} 
=0,
\end{split}
\]
\item ${\rm rank}(A)=3$, where  $A$ is the $3\times n$- matrix whose rows are $(a_1,\ldots,a_n)$, $(b_1,\ldots,b_n)$ and $(c_1,\ldots,c_n)$.
\end{enumerate}
Then 
\[
 |{\rm Epi}(G,\U_4(\F_p))|=N p^{3n}
\]
\end{prop}
\begin{proof}
By the assumption we see that $R\subseteq S_{(3)}$. Hence  by  Proposition~\ref{prop:counting Epi} we have 
\[
|{\rm Epi}(G,\U_4(\F_p))|=|{\rm TMP}(G,\F_p)| p^{3n}.
\]
Thus it is enough to show that $N=|{\rm TMP(G,\F_p)}|$.

To this end, let $\chi_1,\ldots,\chi_n$ be the dual basis to $x_1,\ldots,x_n$ of $H^1(S,\F_p)=H^1(G,\F_p)$, i.e., $\chi_i(x_j)=\delta_{ij}$. 
Let $\varphi=(\alpha,\beta,\gamma)$ be any element in $H^1(G,\F_p)^3$. We write $\alpha=\sum_{i=1}^na_i\chi_i$, $\beta=\sum_{i=1}^nb_i\chi_i$, $\gamma=\sum_{i=1}^nc_i\chi_i$, where $a_i,b_i,c_i\in \F_p$.  Then $\langle\alpha,\beta,\gamma\rangle=0$ if and only if for each $1\leq m\leq r$, we have ${\rm tr}_{\rho_m}(\langle\alpha,\beta,\gamma\rangle)=0$. By \cite[Appendix]{Vo2}, we see that
\[ 
\begin{aligned}
{\rm tr}_{\rho_m}(\langle\alpha,\beta,\gamma\rangle) = & \sum_{i,j,k}a_ib_jc_k\langle \chi_i,\chi_j,\chi_k\rangle\\
=& \sum_{i,j,k} a_ib_jc_k \epsilon_{(ijk),p}(\rho_m)\\
=& \sum_{i<j,k<j,i\not=k} (a_ib_jc_k-a_jb_ic_k+ a_kb_jc_i-a_kb_ic_j)e_{i,j,k,m} \\
&+\sum_{i=k<j,k<j} (2a_ib_jc_k-a_jb_ic_k-a_kb_ic_j)e_{i,j,k,m} \\
&+\sum_{i<j=k} (a_ib_jc_k-2a_jb_ic_k+ a_kb_jc_i)e_{i,j,k,m}.
\end{aligned}
\]
(Here the coefficients $\epsilon_{(ijk),p}(\rho_m)$ are coefficients in the Magnus expansion of $\rho_m$, see \cite{Vo2}.)

Note also that $\alpha,\beta,\gamma$ are $\F_p$-linearly independent if and only if ${\rm rank}(A)=3$. Hence $\varphi=(\alpha,\beta,\gamma)$ is in ${\rm TMP}(G,\F_p)$ if and only if the $3n$-tuple $(a_1,\ldots,a_n,b_1,\ldots,b_n,c_1,\ldots,c_n)$ satisfies two conditions (1)-(2) in the proposition. Therefore $N=|{\rm TMP}(G,\F_p)|$, and the statement follows.
\end{proof}

 Recall that for  prime numbers $p_1,p_2,p_3$ with ${\rm gcd}(p_1,p_2,p_3)=1$, $p_i\equiv 1\bmod 4$, $i=1,2,3$ and 
\[
  \left( \dfrac{p_i}{p_j}\right)=1 \text{  if $p_i\not=p_j$ and $1\leq i,j\leq 3$},
\]
one can define the R\'edei symbol $[p_1,p_2,p_3]$ taking values $\pm1$ as follows.  There exists an element $\alpha\in K_1:=\Q(\sqrt{p_1})$ with the following properties: 
\begin{enumerate}
\item ${\rm Nm}_{K_1/\Q}(\alpha)=p_2$ and 
\item ${\rm Nm}_{K_1/\Q}(D_{K_1(\sqrt{\alpha})/K_1})=p_2$, where $D_{K_1(\sqrt{\alpha})/K_1}$ is the discriminant of the extension  $K_1(\sqrt{\alpha})/K_1$. 
\end{enumerate}
Furthermore, there exists a prime ${\mathfrak{p}_3}$ in $K_1$ over $p_3$ such that $\mathfrak{p}_3$ is unramified in $K_1(\sqrt{\alpha})$. Then the R\'edei symbol is defined as
\[
[p_1,p_2,p_3]:=\begin{cases}
1 & \text{ if $\mathfrak{p}_3$ splits in $K_1(\sqrt{\alpha})$},\\
-1 & \text{ if $\mathfrak{p}_3$ is inert in $K_1(\sqrt{\alpha})$}.
\end{cases}
\]
The value $[p_1,p_2,p_3]$ does not depend on the choices of $\alpha$ and $\mathfrak{p}_3$. (For more details, see for example \cite{Vo2}.)

Using the work \cite{Vo2}, we obtain the following result.
\begin{cor}
Let $S$ be a set of odd primes $\{l_1,\ldots,l_n\}$, where $l_i\equiv 1\mod 4$, $i=1,\ldots,n$, and assume that 
\[
\left(\frac{l_i}{l_j}\right)=1 \text{ for all } 1\leq i,j\leq n, i\not=j.
\]
For $1\leq i<j\leq n$, $k\leq j$, $1\leq m\leq n$, we define $e_{i,j,k,m}$ which will be $0$ or $1$ by the condition that
\[
(-1)^{e_{i,j,k,m}}=
\begin{cases}
 [l_i,l_j,l_k] &\text{ if $m=j$ and $m\not=k$},\\
 [l_i,l_j,l_k] &\text{ if $m\not=j$ and $m=k$},\\
  [l_i,l_j,l_k] &\text{ if $m=i$ and $j=k$},\\
   [l_i,l_j,l_k] &\text{ if $m=j=k$},\\
   1 &\text{ otherwise}.
\end{cases}
\]
Let $N$ be the number of  $(a_1,\ldots,a_n,b_1,\ldots,b_n,c_1,\ldots,c_n)\in \F_2^{3n}$ which satisfy the following two conditions:
\begin{enumerate}
\item For each $1\leq m\leq r$, we have
\[
\begin{split}
\sum_{i<j,k\leq j} (a_ib_jc_k + a_jb_ic_k+ a_kb_jc_i +a_kb_ic_j)e_{i,j,k,m} 
=0,
\end{split}
\]
\item ${\rm rank}(A)=3$, where  $A$ is the $3\times n$- matrix whose rows are $(a_1,\ldots,a_n)$, $(b_1,\ldots,b_n)$ and $(c_1,\ldots,c_n)$.  
\end{enumerate}
Then the number of Galois $\U_4(\F_2)$-extensions over $\Q$ which are unramified outside $\{S\}\cup \{\infty\}$ is $N2^{3n-7}/3$.
\end{cor}
\begin{proof}
Let $G:=G_\sS(2)$ be the Galois group of the maximal 2-extension of $\Q$ unramified outside $\sS\cup \{\infty\}$. By \cite[Theorem 3.12]{Vo2}, $G$ has a presentation $G=S/R$ where $S$ is a free pro-2-group on $n$ generators $x_1,\ldots,x_n$, and $R$ is a normal subgroup generated by $n$ relations $\rho_1,\ldots,\rho_m$, and 
\[
\rho_m\equiv \prod_{1\leq i<j\leq n; k\leq j} [[x_i,x_j],x_k]^{e_{i,j,k,m}} \mod S_{(4)}, \;\forall m=1,\ldots, n. 
\]
Hence by Proposition~\ref{prop:restricted ramification}, $|{\rm Epi}(G,\U_4(\F_2))|=N 2^{3n}$. Therefore the number of Galois $\U_4(\F_2)$-extensions over $\Q$ which are unramified outside $\sS\cup \{\infty\}$ is 
\[
\frac{|{\rm Epi}(G,\U_4(\F_2))|}{|{\rm Aut}(\U_4(\F_2))|}=\frac{N2^{3n}}{3\cdot 2^7}=N2^{3n-7}/3.
\]
(Here by \cite[Theorem]{M}, we know that $|{\rm Aut}(\U_4(\F_2))| =3\cdot 2^7$.)
\end{proof}
\begin{ex}(cf. \cite[Example 3.15]{Vo2}).
\label{ex:ram01}
 Let $\sS=\{5,101,8081,\infty\}$ and let $G:=G_\sS(2)$ be the Galois group of the maximal 2-extension of $\Q$ unramified outside $\sS$. Then $G$ has a minimal presentation 
\[
 1\to R\to S\to G\to 1,
\]
where $S$ is a free pro-$2$-group on generators $x_1,x_2,x_3$, and $R\subseteq S_{(4)}$. Therefore the number of Galois $\U_4(\F_2)$-extensions over $\Q$ which are unramified outside $\{5,101,8081,\infty\}$ is 
\[
\frac{|{\rm Epi}(G,\U_4(\F_2))|}{|{\rm Aut}(\U_4(\F_2))|}=\frac{2^{9}(2^3-1)(2^3-2)(2^3-2^2)}{3\cdot 2^7}=7\cdot 2^{5}=224.
\]
\end{ex}

\begin{ex}(cf. \cite[Example 3.14]{Vo2}). 
\label{ex:ram02}
Let $\sS=\{13,61,937,\infty\}$ and let $G:=G_\sS(2)$ be the Galois group of the maximal 2-extension of $\Q$ unramified outside $\sS$. Then $G$ has a minimal presentation 
\[
 1\to R\to S\to G\to 1,
\]
where $S$ is a free pro-$2$-group on generators $x_1,x_2,x_3$, and a minimal generating system of $R$ as a normal subgroup of $S$ is $\sR=\{\rho_1,\rho_2,\rho_3\}$ with
\[
 \begin{aligned}
  \rho_1 &\equiv [[x_2,x_3],x_1] \bmod S_{(4)},\\
\rho_2 &\equiv [[x_1,x_3],x_2] \bmod S_{(4)},\\
\rho_3 &\equiv [[x_1,x_3],x_2][[x_2,x_3],x_1]\equiv [[x_1,x_2],x_3] \bmod S_{(4)}.
 \end{aligned}
\]
We shall compute $|{\rm Epi}(G,\U_4(\F_2))|$. By Lemma~\ref{lem:equiv}, we can assume that $R$ is generated as a normal subgroup of $S$ by $r_1,r_2$, where
\[
r_1=[[x_2,x_3],x_1], \text{ and } r_2=[[x_1,x_3],x_2]. 
\]
Let $\chi_1,\chi_2,\chi_3$ be the dual basis to $x_1,x_2,x_3$ of $H^1(S,\F_2)=H^1(G,\F_2)$, i.e., $\chi_i(x_j)=\delta_{ij}$.
Then we have (see \cite[Appendix]{Vo2})
\[
 \begin{aligned}
  {\rm tr}_{r_1}\langle \chi_1,\chi_2,\chi_3\rangle&={\rm tr}_{r_1}\langle \chi_3,\chi_2,\chi_1\rangle= {\rm tr}_{r_1}\langle \chi_1,\chi_3,\chi_2\rangle={\rm tr}_{r_1}\langle \chi_2,\chi_3,\chi_1\rangle=1,\\
{\rm tr}_{r_2}\langle \chi_1,\chi_3,\chi_2\rangle&={\rm tr}_{r_2}\langle \chi_2,\chi_3,\chi_1\rangle= {\rm tr}_{r_2}\langle \chi_3,\chi_1,\chi_2\rangle={\rm tr}_{r_2}\langle \chi_2,\chi_1,\chi_3\rangle=1,\\
{\rm tr}_{r_1}\langle \chi_i,\chi_j,\chi_k\rangle&={\rm tr}_{r_2}\langle \chi_i,\chi_j,\chi_k\rangle=0 \text{ for all others $(i,j,k)$}.
 \end{aligned}
\]

Let $\varphi=(\alpha,\beta,\gamma)$ be any element in $H^1(G,\F_2)$. We write $\alpha=\sum_{i=1}^3a_i\chi_i$, $\beta=\sum_{i=1}^3b_i\chi_i$, $\gamma=\sum_{i=1}^3c_i\chi_i$, where $a_i,b_i,c_i\in \F_2$. Then we have
\[
\begin{aligned}
   {\rm tr}_{r_1}\langle\alpha,\beta,\gamma\rangle &= \sum_{i,j,k}a_ib_jc_k {\rm tr}_{r_1}\langle \chi_i,\chi_j,\chi_k\rangle=a_1b_2c_3+a_3b_2c_1+a_1b_3c_2+a_2b_3c_1,\\
   {\rm tr}_{r_2}\langle\alpha,\beta,\gamma\rangle &= \sum_{i,j,k}a_ib_jc_k {\rm tr}_{r_2}\langle \chi_i,\chi_j,\chi_k\rangle=a_1b_3c_2+a_2b_3c_1+a_3b_1c_2+a_2b_1c_3.
\end{aligned}
 \]
Let $A$ be the matrix  with $[a_1,a_2,a_3], [b_1,b_2,b_3], [c_1,c_2,c_3]$ as its rows. 
Then $\varphi\in {\rm TMP}(G,\F_p)$ if and only if $\det A\not=0$ and $\langle\alpha,\beta,\gamma\rangle=0$, if and only if 
\[
\begin{aligned} 
a_1b_2c_3+a_3b_2c_1+a_1b_3c_2+a_2b_3c_1+a_2b_1c_3+a_3b_1c_2&=\det(A)= 1\\
 a_1b_2c_3+a_3b_2c_1+a_1b_3c_2+a_2b_3c_1&=0\\
a_1b_3c_2+a_2b_3c_1+a_3b_1c_2+a_2b_1c_3&=0.
\end{aligned}
\]
This system is equivalent to
\[
\begin{aligned} 
 b_1(a_2c_3+a_3c_2)&=1\\
b_2(a_1c_3+a_3c_1)&=1\\
b_3(a_1c_2+a_2c_1)&=1.
\end{aligned}
\] 
 The number of solutions of the above system is $6$.  
Therefore by Proposition~\ref{prop:R small}, we  have $|{\rm Epi}(G,\U_4(\F_2))|=3\cdot 2^{10}$.
This, together with \cite[Theorem, (2)]{M}, implies that the number of Galois $\U_4(\F_2)$-extensions of $\Q$ unramified outside $\{13,61,937,\infty\}$ is
\[
\frac{|{\rm Epi}(G,\U_4(\F_2))|}{|{\rm Aut}(\U_4(\F_2))|}=\frac{3\cdot 2^{10}}{3\cdot 2^7}=8.
\]
In \cite{Vo2}, $(13,61,937)$ is called a triple of proper Borromean primes modulo 2. 
 \qedhere
\end{ex}

\section{The number of Galois $\U_4(\F_p)$-extensions of a local field}
\subsection{Demushkin groups}
Recall that a pro-$p$-group $G$ is said to be a Demushkin group if
\begin{enumerate}
\item $\dim_{\F_p} H^1(G,\F_p)<\infty,$ 
\item $\dim_{\F_p} H^2(G,\F_p)=1,$
\item  the cup product $H^1(G,\F_p)\times H^1(G,\F_p)\to H^2(G,\F_p)$ is a non-degenerate bilinear form.
\end{enumerate}
(In some literature related to Demushkin groups, condition (1) is relaxed to allow also countable infinite rank.)

We assume now that $G$ is an infinite Demushkin group. Then $G$ is a finitely generated topological group with $d(G) = \dim_{\F_p} H^1(G, \F_p)$ as the minimal number of topological generators. Condition (2) means that there is only one relation among a minimal system of generators for $G$; that is, we can find a minimal presentation of $G$ as above such that $S$ is a free pro-$p$-group of rank $d = d(G)$ on generators $x_1,x_2,\ldots,x_d$, and $R= \langle r\rangle$ is the closed
normal subgroup of $S$ generated by an element $r\in S^p[S,S]$. Hence $G/[G,G]$ is isomorphic to $(\Z_p)^{d-1} \times (\Z_p/qZ_p)$, where $q = q(G)$ is a uniquely determined power of $p$. (By convention $p^\infty=0$.) For convenience, we call $q(G)$ the $q$-invariant of Demushkin group $G$.
Condition (3) then implies that the bilinear form induced by cup product $(\cdot,\cdot)\colon H^1(G,\F_p)\times H^1(G,\F_p)\to H^2(G,\F_p)\stackrel{{\rm tr}_r}{\simeq}\F_p$ is non-degenerate. Note also that $(\cdot,\cdot)$ is skew-symmetric. From the classification of Demushkin groups \cite{La}, the relation $r$ takes the following form:
\begin{enumerate}
\item[(a)] if $q=q(G)\not=2$ ($d$ is even in this case), then 
\begin{equation}
 \label{eq:D1} r=x_1^q[x_1,x_2][x_3,x_4]\cdots[x_{d-1},x_d];
\end{equation}
\item[(b)] if $q=2$ and $d$ is odd, then 
\begin{equation}
 \label{eq:D2} r=x_1^2x_2^{2^f}[x_2,x_3][x_4,x_5]\cdots[x_{d-1},x_d],
 \end{equation}
 where $f$ is an integer $\geq 2$ or $\infty$;
\item[(c)] if $q=2$ and $d$ is even, then either
\begin{equation}
 \label{eq:D3} r=x_1^{2+2^f}[x_1,x_2][x_3,x_4]\cdots[x_{d-1},x_d],\\
\end{equation}
where $f$ is an integer $\geq 2$ or $\infty$, or 
\begin{equation}
\label{eq:D4} r=x_1^2[x_1,x_2]x_3^{2^f}[x_3,x_4]\cdots[x_{d-1},x_d],
\end{equation}
where $f$ is an integer $\geq 2$.
\end{enumerate}
Recall that a bilinear form on a vector space $V$ over a field $K$, $(\cdot,\cdot)\colon V\times V\to K$ is skew-symmetric if $(x,y)=-(y,x)$ for all $x,y\in V$.
We obtain the following result.
\begin{prop}
\label{prop:Demushkin-linear algebra}
 Let $G$ be a Demushkin group with $d(G)=d$ and $q=q(G)$. Let 
\[(\cdot,\cdot)\colon H^1(G,\F_p)\times H^1(G,\F_p)\stackrel{\cup}{\to} H^2(G,\F_p)\stackrel{{\rm tr}_r}{\simeq}\F_p\] be the non-degenerate skew-symmetric bilinear form induced by cup product.  
\begin{enumerate}
 \item If $q\not=2$ then there exists an $\F_p$-basis $v_1,v_2,\ldots,v_d$ of $\H^1(G,\F_p)$ such that $(v_i,v_i)=0$ for every $1\leq i\leq d$.
 \item If $q=2$ then there exists an $\F_p$-basis $v_1,v_2,\ldots,v_d$ of $\H^1(G,\F_p)$ such that $(v_1,v_1)=1$, and that $(v_i,v_i)=0$ for every $2\leq i\leq d$.
\end{enumerate}
\end{prop}
\begin{proof} 
We already observed above that the form $(\cdot,\cdot)$ is non-degenerate, skew-symmetric and bilinear. 

Let $v_1,v_2,\ldots,v_d$ be the dual basis to $x_1,x_2,\ldots,x_d$ of $H^1(S,\F_p)=H^1(G,\F_p)$. That means that $v_i(x_j)=\delta_{ij}$, where $\delta_{ij}\in\{0,1\}$  is the Kronecker $\delta$ function.  
If $p>2$ then $(v_i,v_i)=0$ for every $i=1,2,\ldots,d$ because our pairing is a skew-symmetric  bilinear form. 

We now assume  that $p=2$. We shall apply \cite[Proposition 3.9.13]{NSW} to the case of descending $2$-central series. Then we see that $(v_1,v_1)=1$ and $(v_i,v_i)=0$ for all $i=2,3,\ldots,d$ in the cases \ref{eq:D2}, \ref{eq:D3} and \ref{eq:D4}. In the case \ref{eq:D1} and $p=2$, we see that $q\geq 4$. Hence we obtain that $(v_i,v_i)=0$ for all $i=1,2\ldots,d$
\end{proof}
\subsection{Euler-Poincar\'e characteristics} 
Let $G$ be a pro-$p$-group. Assume that $H^i(G,\F_p)$ has finite dimension over $\F_p$ for each $i$. Further assume that $G$ has a finite cohomological dimension $c$. 

Let $A$  be a finite (discrete) $G$-module  of  $p$-power order. Then we also have $\dim_{\F_p} H^i(G,A)<+\infty$ for each $i$ and $\dim_{\F_p} H^i(G,A)=0$ for  $i>c$ (see \cite[Section 5.1]{Ko}). We define
\[
\chi(G,A)=\sum_{i=0}^\infty (-1)^i\dim_{\F_p} H^i(G,A).
\]
When $A=\F_p$ with trivial $G$-action, $\chi(G,\F_p)$ is the {\it Euler-Poincar\'e characteristic} of $G$. 

For a $G$-module $A$ of order $p^r$, one has (see \cite[Section 5.2]{Ko})
\[
\chi(G,A)=r \chi(A,\F_p).
\]
\begin{lem}
\label{lem:EP}
  Let $G$ be a  Demushkin pro-$p$-group of rank $d$. Let $A$ be a finite $G$-module of order $p^r$. Then
\[
\chi(G,A)=r(2-d).
\]
\end{lem}
\begin{proof}
Since $G$ is a Demushkin group, $G$ is a Poincar\'e group of dimension 2. This implies  in particular that $G$ is cohomological dimension 2, and that $\dim_{\F_p} H^2(G,\F_p)=\dim_{\F_p}H^0(G,\F_p) =1$. Hence $\chi(A,\F_p)=2-d$, and the statement follows.
\end{proof}
\begin{rmk}
Let $K$ be a local field which is a finite extension of degree $n$ of $\Q_p$. Assume that $K$ contains a primitive $p$-th root of unity.  Then the Galois group of a maximal $p$-extension of $K$ is a Demushkin group of rank $n+2$. Lemma~\ref{lem:EP} implies the following result, which is a special but important case of a result of Tate. (See \cite[Chapter II, \S 5.7, Theorem 5 and \S 5.6 Proposition 20]{Se}, \cite[Theorem 7.3.1]{NSW}.)
\begin{thm}[Tate]
\label{thm:Tate} If the order of $A$ is $p^r$, then we have 
\[
\frac{|H^0(K,A)|\cdot |H^2(K,A)|}{|H^1(K,A)|}=p^{-rn}.
\]
\end{thm}
\end{rmk}
\subsection{Cohomology of certain modules} In this subsection we assume that $G$ is a Demushkin pro-$p$-group of rank $d$ with $d\geq 3$.

Recall that we have the following exact sequence of finite groups
\[
1 \longrightarrow M \longrightarrow \U_4(\F_p) \xrightarrow{(a_{12},a_{23},a_{34})} (\F_p)^3\longrightarrow 1,
\]
where $a_{ij}\colon\U_4(\F_p)\to\F_p$ is the map sending a matrix to its $(i,j)$-coefficient. This exact sequence induces a $B:=(\F_p)^3$-module structure on $M$.

Now let $\varphi\colon G\to B$ be any continuous homomorphism. We consider $M$ as a $G$-module, denoted $M_\varphi$, via $\varphi$ and the action of $B$ on $M$. Note that $pM=\{0\}$.
\begin{lem}
\label{lem:counting Z1}
\mbox{}
\begin{enumerate}
\item If $\im \varphi \subseteq \{0\}\times \F_p\times \{0\}$ then $|Z^1(G,M_\varphi)|=p^{3d}$.
\item If $\im \varphi \not\subseteq \{0\}\times \F_p\times \{0\}$ then $|Z^1(G,M_\varphi)|=p^{3d-1}$.
\end{enumerate}
\end{lem}
\begin{proof} For simplicity, we shall write $M$ for $M_\varphi$. 
We  note that $\ker(d\colon M\to C^1(G,M))=H^0(G,M)$. Hence $|B^1(G,M)|=|\im(d)|=|M/H^0(G,M)|$. Therefore 
\begin{equation}
\label{eq:Z1}
\begin{aligned}
 \dim_{\F_p} Z^1(G,M)| &=\dim_{\F_p} H^1(G,M)+\dim_{\F_p}B^1(G,M)\\
 &= \dim_{\F_p} H^1(G,M)-\dim_{\F_p}H^0(G,M)+\dim_{\F_p}M\\
 &=\dim_{\F_p}H^2(G,M)+3-\chi(G,M)\\
 &=\dim_{\F_p}H^2(G,M) +3d-3.
 \end{aligned}
 \end{equation}
(Since $|\U_4(\F_p)|=p^6$, we see that $\dim_{\F_p}(M)=3$. The last equality follows from Lemma~\ref{lem:EP}.)

Let $M^\prime={\rm Hom}(M,\F_p)$ be the dual $B$-module of the $B$-module $M$. If we consider $M^\prime$ as a $G$-module via the map $\varphi\colon G\to B\to {\rm Aut}(M^\prime)$, then $M^\prime$ is the dual $G$-module of the $G$-module $M$. 
The  duality theorem for 2-dimensional Poincar\'e groups implies that $|H^2(G,M)|=|H^0(G,M^\prime)|$ (see \cite[Chapter I, \S4, Proposition 30]{Se}).

By \cite[Lemma 3.2]{MT2}, there is an $\F_p$-basis of $M^\prime$ such that with respect to this basis the structure map $\alpha\colon B\to {\rm Aut}(M^\prime)$ becomes the map $\alpha\colon B\to {\rm GL_3}(\F_p)$, which sends $(x,y,z)\in B$ to 
$\begin{bmatrix}
1 & 0 &-x\\
0& 1 & z\\
 0& 0 &1
\end{bmatrix}.$ It is then straightforward to check that 
\[
H^0(G,M^\prime)\simeq 
\begin{cases}
\{ \begin{bmatrix} a\\ b \\ c\end{bmatrix} \mid a,b,c\in \F_p\} \;\text{ if }\im \varphi \subseteq \{0\}\times\F_p\times\{0\}\\\\
\{ \begin{bmatrix} a\\ b \\ 0\end{bmatrix} \mid a,b\in \F_p\} \; \text{ if }\im \varphi\not \subseteq \{0\}\times\F_p\times\{0\}.
\end{cases}
\]
Hence $\dim_{\F_p}H^2(G,M)=\dim_{\F_p} H^0(G,\F_p)=3$ (if we are in part (1)), or $2$ (if we are in part (2)). The lemma then follows from Equation \ref{eq:Z1}
\end{proof}
\subsection{A result in linear algebra}
\begin{lem}
\label{lem:LA}
Let $V$ be an $\F_p$-vector space of dimension $d\geq 3$ with basis $v_1,v_2,\ldots,v_d$. Let $(\cdot,\cdot)\colon V\times V\to \F_p$ be a non-degenerate skew-symmetric bilinear form on $V$. Let $N$ be the number of triples $(x,y,z)\in V\times V\times V$ such that $(x,y)=(y,z)=0$ and that $x,y,z$ are $\F_p$-linearly independent.
\begin{enumerate}

\item If $(v_i,v_i)=0$ for every $1\leq i\leq d$, then 
\[N=(p^d-1)(p^{d-1}-p)(p^{d-1}-p^2).\]
\item If $(v_1,v_1)=1$ and $(v_i,v_i)=0$ for every $2\leq i\leq d$, then $p=2$ and
\[N=(2^{d-1}-1)(2^{d-1}-2)(2^d-4).\]
\end{enumerate}
\end{lem}
\begin{proof}
 For each $y\in V\setminus \{0\}$, we denote 
\[
 y^\perp =\{ x\in V\mid (x,y)=0\}.
\]
Then $y^\perp$ is an $\F_p$-vector space, and $\dim y^\perp=\dim V -1= d-1$ since the pairing $(\cdot,\cdot)$ is non-degenerate. We set 
\[
\begin{aligned}
 M(y)&:=\{(x,z)\in V\times V\mid (x,y)=(y,z)=0; x,y,z \text{ are $\F_p$-linearly independent}\}\\
&=\{(x,z)\in y^\perp\times y^\perp\mid x,y,z \text{ are $\F_p$-linearly independent}\}.
\end{aligned}
\]

(1) We claim that $y\in y^\perp$. Indeed, writing $y=\sum_{i\in I} a_i v_i$, $a_i\in \F_p$, $I\subseteq \{1,2,\ldots,d\}$, then 
\[
\begin{aligned}
 (y,y)&=(\sum_{i\in I} a_i v_i, \sum_{j\in I} a_j v_j)=\sum_{i,j\in I}a_ia_j(v_i,v_j)\\
&=\sum_{i\in I} a_i^2(v_i,v_i)+\sum_{\substack{ i\not=j,\\ i,j\in I}}a_ia_j(v_i,v_j)=0,
\end{aligned}
\]
because $(\cdot,\cdot)$ is skew-symmetric and $(v_i,v_i)=0$ for $1\leq i\leq d-1$ by assumption. Thus 
\[
 |M(y)|=(p^{d-1}-p)(p^{d-1}-p^2).
\]
Therefore
\[
 N=\sum_{y\in V\setminus\{0\}} |M(y)|=(p^d-1)(p^{d-1}-p)(p^{d-1}-p^2),
\]
as desired.
\\
\\
(2) In this case $p=2$. Let us write $y=\sum_{i\in I}v_i$, $I\subseteq \{1,2,\ldots,d\}$.

{\bf Case 1:} $1\not\in I$. Then we claim that $y\in y^\perp$. Indeed
\[
\begin{aligned}
 (y,y)&=(\sum_{i\in I} v_i, \sum_{j\in I} v_j)=\sum_{i,j\in I}(v_i,v_j)\\
&=\sum_{i\in I} (v_i,v_i)+\sum_{\substack{ i\not=j,\\ i,j\in I}}(v_i,v_j)=0,
\end{aligned}
\]
because $(\cdot,\cdot)$ is skew-symmetric and $(v_i,v_i)=0$ for $2\leq i\leq d-1$ by assumption. Thus 
\[
 |M(y)|=(2^{d-1}-2)(2^{d-1}-4).
\]
{\bf Case 2:} $1\in I$. Then we claim that $(y,y)=1$, and hence $y\not\in y^\perp$. Indeed, let $I^\prime:=I\setminus\{1\}$, then
\[
\begin{aligned}
 (y,y)&=(v_1+\sum_{i\in I^\prime} v_i, v_1+\sum_{i\in I^\prime} v_j)\\
&=(v_1,v_1)+ (v_1,\sum_{i\in I^\prime} v_i)+(\sum_{i\in I^\prime} v_i,v_1)+(\sum_{i\in I^\prime} v_i,\sum_{i\in I^\prime} v_i)\\
&=(v_1,v_1)=1,
\end{aligned}
\]
because $(\cdot,\cdot)$ is skew-symmetric, and $(\sum_{i\in I^\prime} v_i,\sum_{i\in I^\prime} v_i)=0$ by Case 1. We note also that $y$ is linearly independent from $y^\perp\setminus\{0\}$. Indeed assume that $a y+b v=0$ for some $a,b\in \F_2$, $0\not=v\in y^\perp$. Then using $0=(ay+bv,y)=a(y,y)+b(v,y)=a$,  we see that both $a$ and hence $b$ are 0.  Thus $|M(y)|$ is equal to the number of pairs $(x,z)\in y^{\perp}$ such that $x,z$ are $\F_2$-linearly independent. Therefore  
\[
 |M(y)|=(2^{d-1}-1)(2^{d-1}-2).
\]

We then obtain
\[
\begin{aligned}
 N&=\sum_\text{$y$ in Case 1} |M(y)|+ \sum_\text{$y$ in Case 2} |M(y)|\\
&=(2^{d-1}-1)(2^{d-1}-2)(2^{d-1}-4) +2^{d-1}(2^{d-1}-1)(2^{d-1}-2)\\
&=(2^{d-1}-1)(2^{d-1}-2)(2^{d}-4),
\end{aligned}
\]
as desired.
\end{proof}
\subsection{The main results}
\begin{thm}
\label{thm:counting U4 Demushkin}
  Let $G$ be a Demushkin group of rank $d$ with $d\geq 1$. Let $q=q(G)$ be its $q$-invariant. Then
\[
|{\rm Epi}(G,\U_4(\F_p))|=
\begin{cases}
(p^{d}-1)(p^{d-1}-p)(p^{d-1}-p^2)p^{3d-1} &\text{ if $q\not=2$},\\
(2^{d-1}-1)(2^{d-1}-2)(2^{d}-4)2^{3d-1} &\text{ if $q=2$}.
\end{cases}
\]
\end{thm}
\begin{proof}
First note that $G$ has the vanishing triple Massey product property (see \cite[Theorem 4.2]{MT1}). The result then follows from Corollary~\ref{cor:VTM}, Proposition~\ref{prop:Demushkin-linear algebra}, Lemma~\ref{lem:counting Z1} and Lemma~\ref{lem:LA}.
\end{proof}
\begin{thm}
\label{thm: local U4-extensions}
 Let $K$ be a local field which is a finite extension of degree $n$ of $\Q_p$. Assume that $K$ contains a primitive $p$-th root of unity. Let $q$ be the highest power of $p$ such that $K$ contains a
primitive $q$-th root of unity.
Then the number of Galois $\U_4(\F_p)$-extensions of $K$ is 
\[
 \begin{cases}
\displaystyle
\frac{(p^{n+2}-1)(p^n-1)(p^{n-1}-1)p^{3n}}{2(p-1)^3} & \text{ if $p\not=2$},\\
\displaystyle
\frac{(2^{n+2}-1)(2^n-1)(2^{n-1}-1)2^{3n+1}}{3} & \text{ if $p=2$ and $q\not=2$},\\
\displaystyle
  \frac{(2^{n+1}-1)(2^n-1)^2 2^{3n+1}}{3} &\text { if $q=2$}.
 \end{cases}
\]
\end{thm}
\begin{proof} Let $G$ be the Galois group of a maximal $p$-extension of $K$. Then it is well-known that $G$ is a Demushkin group of rank $d=n+2$ and $q(G)=q$ (\cite{La}).
Note also that the number of Galois $\U_4(\F_p)$-extension of $K$ is 
$\displaystyle
\frac{|{\rm Epi}(G,\U_4(\F_p))|}{|{\rm Aut}(\U_4(\F_p))|}.
$
The result then follows from Theorem~\ref{thm:counting U4 Demushkin}  and from the fact that 
\[
|{\rm Aut}(\U_4(\F_p))|=
\begin{cases} 
3\cdot 2^7 & \text{ if $p=2$},\\
2(p-1)^3p^8 &\text{ if $p>2$}.
\end{cases}
\]
(This latter fact follows from \cite{M} for $p=2$, and from \cite{Pa} for $p>2$.)
\end{proof}
\begin{rmk} 
\label{rmk:CP}
In this remark we observe that our methods work efficiently in the case $n =3$, and in fact we can recover some results originally obtained in \cite[Theorem 11]{MNg} and reproved and extended in \cite[Section 2]{Ya}.

Let the notation be as in Theorem~\ref{thm: local U4-extensions}. By considering the following exact sequence
\[
1\to \F_p\to \U_3(\F_p)\to \F_p\times \F_p\to 1,
\]
instead of the exact sequence (\ref{eq:U4 exact sequence}), we can derive  the number $\nu(K,\U_3(\F_p))$ of Galois $\U_3(\F_p)$-extensions of $K$ as follows.  Let $G:=G_K(p)$. Then $G$ is a Demushkin group of rank $n+2$.
Let ${\rm CP}(G,\F_p)$ is the set of $(x,y)\in H^1(G,\F_p)\times H^1(G,\F_p)$ such that $x\cup y=0$ and that $x,y$ are $\F_p$-linearly independent in $H^1(G,\F_p)$.  For each $\varphi \in {\rm CP}(G,\F_p)$, let us still denote, by abuse of notation, $\varphi\colon G\to \F_p^2$ the induced surjective homomorphism. An analogous result to Proposition~\ref{prop:counting Epi} is that
\[
|{\rm Epi}(G,\U_3(\F_p))|=\sum_{\varphi\in {\rm CP}(G,\F_p)}|Z^1(G,\F_p)| =|{\rm CP}(G,\F_p)|\cdot p^{n+2}.
\]
From Proposition~\ref{prop:Demushkin-linear algebra} and the proof of Lemma~\ref{lem:LA} we see that
\[
|{\rm CP}(G,\F_p)| =
\begin{cases}
(p^{n+2}-1)(p^{n+1}-p) & \text{ if $p>2$},\\
(2^{n+2}-1)(2^{n+1}-2) & \text{ if $p=2$ and $q>2$},\\
(2^{n+1}-1)(2^{n+1}-2)+2^{n+1}(2^{n+1}-1) &\text { if $p=2$ and $q=2$}.
 \end{cases}
\]
Note also that 
\[
|{\rm Aut}(\U_3(\F_p))|= 
\begin{cases}
p^3(p^2-1)(p-1) &\text{ if $p>2$},\\
8 &\text{if $p=2$}. 
\end{cases}
\]
Therefore
\[
\nu(K,\U_3(\F_p))=\frac{|{\rm Epi}(G,\U_3(\F_p))|}{|{\rm Aut}(\U_3(\F_p))|}=
 \begin{cases}
\displaystyle
\frac{p^n(p^{n+2}-1)(p^n-1)}{(p^2-1)(p-1)} & \text{ if $p>2$},\\
2^n(2^n-1)(2^{n+2}-1) & \text{ if $p=2$ and $q>2$},\\
\displaystyle
2^n(2^{n+1}-1)^2 &\text { if $p=2$ and $q=2$}.
 \end{cases}
 \qedhere
\]
\end{rmk}
\begin{rmk} If $K$ is a finite extension of $\Q_2$ and $p=2$, then the work \cite{AMT} provides another method to count the number of Galois $\U_4(\F_2)$-extensions over $K$. In particular \cite{AMT} provides an explicit list of all 16 such extensions over $\Q_2$. 
\end{rmk}
\section{Epimorphisms from Demushkin groups to $\U_n(\F_p)$} 
In this section we shall provide a necessary and sufficient condition to ensure that ${\rm Epi}(G,\U_n(\F_p))$ is not empty, where $G$ is a Demushkin group, see Theorem~\ref{thm: Demushkin-Un}. In the course of  showing  this result, we establish a result on weakening the defining condition for a Massey product for a Demushkin group, Proposition~\ref{prop: Demushkin n-Massey}, which we hope  is interesting in its own.

For some convenience, we introduce the following definition.
Let $n\geq 1$ be an integer. Let $a_1,\ldots,a_n$ be elements in $H^1(G,\F_p)$.
A collection $\sM=\{a_{ij}\mid 1\leq i<j\leq n+1\}$ of elements $a_{ij}$ of $\sC^1(G,\F_p)$ is called a {\it complete defining system} for the $n$-tuple $(a_1,\ldots,a_n)$ if the following conditions are fulfilled:
\begin{enumerate}
\item $a_{i,i+1}=a_i$ for all $i=1,2\ldots,n$.
\item $\delta a_{ij}= \sum_{l=i+1}^{j-1} a_{il}\cup a_{lj}$ for  all $i+1<j$.
\end{enumerate}
Note that for $n=1$, $\sM=\{a_{12}:=a_1\}$ is a complete defining system for $(a_1)$. For $n=2$, $(a_1,a_2)$ has a complete defining system if and only if $a_1\cup a_2=0$. For $n\geq 3$, $(a_1,\ldots,a_n)$ has a complete defining system if and only if the $n$-fold Massey product $\langle a_1,\ldots,a_n\rangle$ is defined and contains 0.

\begin{prop}
\label{prop: Demushkin n-Massey}
Let $G$ be a Demushkin pro-$p$-group and let $n$ be a natural number $\geq 3$. Let $\chi_1,\ldots, \chi_n$ be $n$  elements in $H^1(G,\F_p)$. Then the $n$-fold Massey product $\langle \chi_1,\ldots,\chi_n\rangle $ is defined if and only $\chi_1\cup\chi_2=\chi_2\cup \chi_3=\cdots=\chi_{n-1}\cup\chi_n=0$.
\end{prop}
\begin{proof} 

From the condition (2) for a defining system of a Massey product in Definition~\ref{defn:Massey product} applying to pairs $(i,j)=(1,3),(2,4),\cdots,(n-1,n+1)$, we see that our condition is necessary. 
It is then enough to prove the ``only if'' statement. We first assume that all $\chi_1,\ldots,\chi_n$ are non-zero. We shall proceed by induction on $n$. For $n=3$ the sufficiency follows as there are no other conditions for a defining system. 
Assume now that $n>3$. By induction hypothesis the $n-1$-fold Massey product $\langle \chi_2,\ldots,\chi_n\rangle$ is defined, and it contains 0 by \cite[Theorem 4.2]{MT1}. Therefore there exists a  defining system $M=\{a_{ij}\in C^1(G,\F_p),2\leq i<j\leq n+1\}$ for $( \chi_2,\ldots,\chi_n)$.
Since $\chi_1\cup \chi_2=0$, there exists $b_{13}\in C^1(G,\F_p)$ such that $\delta b_{13}= a_{12}\cup a_{23}$.  Since $\chi_3\not=0$, the map $\chi_3\cup (-)\colon H^1(G,\F_p)\to H^2(G,\F_p)\simeq \F_p$ is nonzero and hence surjective. Thus there exists $c_{13}\in Z^1(G,\F_p)$ such that 
\[c_{13}\cup a_{34}=-(a_{12}\cup a_{24}+b_{13}\cup a_{34}) \bmod  B^2(G,\F_p).\]
 Setting $a_{13}:=b_{13}+c_{13}$, then 
\[
\delta a_{13}=a_{12}\cup a_{23},
\]
and $a_{12}\cup a_{24}+a_{13}\cup a_{34}=\delta b_{14}$ for some $b_{14}\in C^1(G,\F_p)$. 

Since $\chi_4\not=0$, the map $\chi_4\cup (-)\colon H^1(G,\F_p)\to H^2(G,\F_p)\simeq \F_p$ is surjective. Thus there exists $c_{14}\in Z^1(G,\F_p)$ such that 
\[ c_{14}\cup a_{45}=-(a_{12}\cup a_{25}+a_{13}\cup a_{35}+b_{14}\cup a_{45}) \bmod  B^2(G,\F_p).\] 
Setting $a_{14}:=b_{14}+c_{14}$, then 
\[
\delta a_{14}=\delta b_{14}=a_{12}\cup a_{24}+a_{13}\cup a_{34},
\]
and $a_{12}\cup a_{25}+a_{13}\cup a_{35}+a_{14}\cup a_{45}=\delta b_{15}$ for some $b_{15}\in C^1(G,\F_p)$. By repeating this argument we can construct $a_{13},a_{14},\cdots,a_{1n}\in C^1(G,\F_p)$ such that
\[
\delta a_{1j}= \sum_{l=2}^{j-1} a_{1l}\cup a_{lj}, \quad \forall 3\leq j\leq n.
\]
Then $M =\{a_{ij}, 1\leq i<j\leq n+1, (i,j)\not=(1,n+1)\}$ is a defining system for the Massey product $\langle \chi_1,\ldots,\chi_n\rangle$, as desired.

Combining with \cite[Theorem 4.2]{MT1}, we see that we have shown a stronger statement that for every $n\geq 3$ if $\chi_i$'s are all non-zero and if $\chi_1\cup\chi_2=\cdots=\chi_{n-1}\cup\chi_n=0$ then $(\chi_1,\ldots,\chi_n)$ has a complete defining system. As discussed right before stating this theorem, this stronger statement remains to hold true for $n=1$ and $2$ also.

We now consider any $n\geq 1$ elements $\chi_1,\ldots,\chi_n$ (not necessarily all non-zero) such that $\chi_1\cup\chi_2=\cdots=\chi_{n-1}\cup\chi_n=0$, we shall show that $(\chi_1,\ldots,\chi_n)$ has a complete defining system. We prove this statement by induction on the number $r$ of zero elements among $\chi_1,\ldots,\chi_n$. If $r=0$, then that means that all $\chi_1,\ldots,\chi_n$ are all non-zero, and we are done by the previous discussion. Assume now that $r\geq 1$, and assume that $\chi_k=0$ for some $1\leq k\leq n$. By induction hypothesis $(\chi_1,\ldots,\chi_{k-1})$ has a complete defining system $\{a_{ij}\in C^1(G,\F_p),1\leq i<j\leq k\}$, and $(\chi_{k+1},\ldots,\chi_n)$ has a complete defining system $\{a_{ij}\in C^1(G,\F_p),k+1\leq i<j\leq n+1\}$. For $1\leq i\leq k$ and $k+1\leq j\leq n+1$, we set $a_{ij}:=0$. Then it is straightforward to check that $\{a_{ij}\in C^1(G,\F_p),1\leq i<j\leq n+1\}$ is a complete defining system for $(\chi_1,\ldots,\chi_n)$.
\end{proof}

In the case when $G$ is a free pro-$p$-group, we have $H^2(G,\F_p)=0$. It is immediate that for all $n\geq 2$ and for all $\chi_1,\ldots,\chi_n\in H^1(G,\F_p)$, the $n$-fold Massey product $\langle \chi_1,\ldots,\chi_n\rangle$ is defined and equal to 0.
In Remark~\ref{rmk:counterex} below we provide two simple examples of  groups $G$ which shows that Proposition~\ref{prop: Demushkin n-Massey} is not true for all pro-$p$-groups. However these groups $G$ in the example can not occur as the Galois group $G_F(p)$ for some field $F$ containing a primitive $p$-th root of unity. This follows from \cite[Corollary 9.2]{CEM} for the first example, and from \cite{Be} for the second example. This leads us to an interesting question.
\begin{question}
\label{question:n-Massey}
Suppose that $G=G_F(p)$ for some field $F$ containing a primitive $p$-th root of unity. Let $n\geq 3$ be an integer, and let $\chi_1,\ldots,\chi_n$ be elements in $H^1(G,\F_p)$. Is it then true that the $n$-fold Massey product 
$\langle \chi_1,\ldots,\chi_n\rangle$ is defined  whenever
\[
0=\chi_1\cup\chi_2=\chi_2\cup\chi_3=\cdots=\chi_{n-1}\cup\chi_n?
\]
\end{question}
\begin{rmk} Let $p$ be an odd prime. Assume that $G=G_F(p)$ for some field $F$ containing a primitive $p$-th root of unity. Let $\chi$ be any element in $H^1(G,\F_p)$. Then $\chi \cup \chi=0$. In \cite[Section 8]{MTE}, for any integer $n\geq 3$, we show that the $n$-fold Massey product $\langle \chi,\ldots,\chi\rangle$ is defined.
\end{rmk}
\begin{rmks}
\label{rmk:counterex}
(1) Let $S$ be a free pro-$p$-group on 4 generators $x_1,x_2,x_3,x_4$ and let $r=[[x_2,x_3],x_1]$. Let $G=S/\langle r\rangle$. Let $v_1,v_2,v_3,v_4$ be the dual basis to $x_1,x_2,x_3,x_4$ of $H^1(S,\F_p)=H^1(G,\F_p)$. Then by \cite[Proposition 3.9.13]{NSW},
\[
v_1\cup v_2 =v_2\cup v_3=v_3\cup v_4=0.
\]
However the 4-fold Massey product $\langle v_1,v_2,v_3,v_4\rangle$ is not defined since the triple Massey product $\langle v_1,v_2,v_3\rangle$ does not contain 0 (see for example \cite[Proposition 7.7]{MT1}).

(2) We can also provide an easier example if we do not require the $\chi_i$'s in the statement of Proposition~\ref{prop: Demushkin n-Massey} are distinct. Let $p$ be an odd prime and let $G$ be the cyclic group $C_p$ of order $p$. Let $\chi\colon G\to\F_p$ be the identity homomorphism. Then $\chi\cup\chi=0$. However the $p+1$-fold Massey product $\langle \chi,\ldots,\chi\rangle$ is not defined. This is because the $p$-fold Massey product is defined but does not contain 0 (see \cite[Example 4.6]{MT1}).
\end{rmks}
In order to investigate Question~\ref{question:n-Massey} in a systematic way, we formulate the following definition.
\begin{defn} Let $G$ be a profinite group and $p$  a prime number. We say that $G$ has the {\it cup-defining $n$-fold Massey product} property (with respect to $\F_p$) if for every $\chi_1,\ldots,\chi_n\in H^1(G,\F_p)$ with 
\[0=\chi_1\cup\chi_2=\chi_2\cup\chi_3=\cdots=\chi_{n-1}\cup\chi_n,\]
the $n$-fold Massey product $\langle \chi_1,\ldots,\chi_n\rangle$ is defined.
\end{defn}
\begin{rmk}
Let $G$ be a profinite group. Clearly, $G$ always has the cup-defining 3-fold Massey product property.  Let $n\geq 4$ be an integer.  Then if $G$ has the cup-defining $n$-fold Massey product property, then $G$ has the vanishing $(n-1)$-fold Massey product property. Indeed, let $\chi_1,\ldots,\chi_{n-1}$ be $n-1$ elements in $H^1(G,\F_p)$ such that the $(n-1)$-fold Massey product $\langle\chi_1,\ldots,\chi_{n-1}\rangle$ is defined. In particular, we have 
\[\chi_1\cup\chi_2=\cdots=\chi_{n-2}\cup\chi_{n-1}=0.\] 
Since $G$ has the cup-defining $n$-fold Massey product property, the $n$-fold Massey product $\langle \chi_1,\ldots,\chi_{n-2},\chi_{n-1},\chi_{n-2}\rangle$ is defined.  This in turn implies that the $(n-1)$-fold Massey product $\langle \chi_1,\ldots,\chi_{n-2},\chi_{n-1}\rangle$ contains 0.
\end{rmk}
In this current subsection we shall use the following convention on notation. Let $G_1$, $G_2$ and $H$ be three pro-$p$-groups. For any homomorphism $\rho \colon G_1*G_2\to H$,  we denote by $^1\rho$ (respectively, $^2\rho$) the composition of the natural homomorphism $G_1\to G_1*G_2$ (respectively, $G_2\to G_1*G_2$ ) with $\rho$. 
The following lemma can be proved by using some basic properties of cup products and free products. (See \cite[Chapter I, \S 4-\S 5 and Theorem 4.1.14]{NSW}, also compare with \cite[Section 2]{MSw}.) However we prefer a proof using Massey-like considerations which nicely fit into the theme of our paper.
\begin{lem} 
\label{lem:cup product}
 Let $u$ and $v$ be elements in $H^1(G_1*G_2,\F_p)={\rm Hom}(G_1*G_2,\F_p)$. Then  $u\cup v=0$ in $H^2(G_1*G_2,\F_p)$ if and only if $\,^1u\cup \,^1v=0$ in $H^2(G_1,\F_p)$ and $\,^2u\cup \,^2v=0$ in $H^2(G_2,\F_p)$.
\end{lem}
\begin{proof} Assume that $u\cup v=0$, there exists a homomorphism $\rho\colon G_1*G_2\to \U_3(\F_p)$ such that $\rho_{12}=u$ and $\rho_{23}=v$. Then we have
\[
\begin{aligned}
(\,^1\rho)_{12}=\,^1(\rho_{12})= \,^1u,\\
(\,^1\rho)_{23}=\,^1(\rho_{23})=\,^1v.
\end{aligned}
\]
This means that $\,^1\rho\colon G_1\to \U_3(\F_p)$ is a lift of $(\,^1u,\,^1v)\colon G_1\to \F_p\times\F_p$. Therefore $\,^1u\cup \,^1v=0$ in $H^2(G_1,\F_p)$.

Similarly, $\,^2u\cup \,^2v=0$ in $H^2(G_2,\F_p)$.

Conversely, assume that $\,^1u\cup \,^1v=0$ in $H^2(G_1,\F_p)$ and $\,^2u\cup \,^2v=0$ in $H^2(G_2,\F_p)$. Then there exist a homomorphism $\rho_1\colon G_1\to \U_3(\F_p)$ and a homomorphism $\rho_2\colon G_2\to \U_3(\F_p)$ such that 
\[
\begin{aligned}
(\rho_1)_{12}=\,^1u,\quad (\rho_1)_{23}=\,^1v,\\
(\rho_2)_{12}=\,^2u, \quad (\rho_1)_{23}=\,^2v.
\end{aligned}
\]

By the universal property of free products, we see that there exists a unique homomorphism $\rho\colon G_1*G_2\to \U_3(\F_p)$ such that $\,^1\rho=\rho_1$ and $\,^2\rho=\rho_2$. Then we have
\[
\begin{aligned}
\,^1(\rho_{12})=(\,^1\rho)_{12}=(\rho_1)_{12}=\,^1u, \\
\,^2(\rho_{12})=(\,^2\rho)_{12}=(\rho_2)_{12}=\,^2u,\\
\,^1(\rho_{23})=(\,^1\rho)_{23}=(\rho_1)_{23}=\,^1v,\\
 \,^2(\rho_{23})=(\,^2\rho)_{23}=(\rho_2)_{23}=\,^2v.
\end{aligned}
\]
This implies that $\rho_{12}=u$ and $\rho_{23}=v$. Therefore $u\cup v=0$ in $H^1(G_1*G_2,\F_p)$.
\end{proof}
\begin{prop}
\label{prop:free product}
Let $G_1$ and $G_2$ be two pro-$p$-groups. Then the free pro-$p$ product $G_1*G_2$ has the cup-defining $n$-fold Massey product property if and only if both $G_1$ and $G_2$ do.
\end{prop}
\begin{proof}Assume that $G_1$ and $G_2$ have the cup-defining $n$-fold Massey product property. Let $\chi_1,\ldots,\chi_n$ be elements in $H^1(G_1*G_2,\F_p)$ such that 
\[
0=\chi_1\cup\chi_2=\cdots=\chi_{n-1}\cup\chi_n.
\]
By Lemma~\ref{lem:cup product}, we have
\[
0=\,^1\chi_1\cup\,^1\chi_2=\cdots=\,^1\chi_{n-1}\cup\,^1\chi_n.
\]
Since $G_1$ has the cup-defining $n$-fold Massey product property, the $n$-fold Massey product $\langle \,^1\chi_1,\ldots,\,^1\chi_n \rangle$ is defined. Therefore there exists a homomorphism $\rho_1\colon G_1\to \bar{\U}_{n+1}(\F_p)$ such that $(\rho_1)_{i,i+1}=-\,^1\chi_i$, for $i=1,\ldots,n$. 
Similarly  there exists a homomorphism $\rho_2\colon G_2\to \bar{\U}_{n+1}(\F_p)$ such that $(\rho_2)_{i,i+1}=-\,^2\chi_i$, for $i=1,\ldots,n$. The universal property of free products implies that there exists  a unique homomorphism $\rho\colon G_1*G_2\to \bar{\U}_{n+1}(\F_p)$ such that $\,^1\rho=\rho_1$ and $\,^2\rho=\rho_2$. Then for each $i=1,\ldots,n$, we have
\[
\begin{aligned}
\,^1(\rho_{i,i+1})=(\,^1\rho)_{i,i+1}=(\rho_1)_{i,i+1}=-\,^1\chi_i,\\
\,^2(\rho_{i,i+1})=(\,^2\rho)_{i,i+1}=(\rho_2)_{i,i+1}=-\,^2\chi_i.
\end{aligned}
\]
Therefore $\rho_{i,i+1}=-\chi_i$. This implies that $\langle \chi_1,\ldots,\chi_n\rangle$ is defined, as desired.

Conversely, assume that $G_1*G_2$ has  the cup-defining $n$-fold Massey product property. Let $u_1,\ldots,u_n$ be elements of $H^1(G_1,\F_p)$ such that
$0=u_1\cup u_2=\cdots=u_{n-1}\cup u_n$. For each $i=1,\ldots,n$ there exists $\chi_i\in H^1(G_1*G_2,\F_p)$ such that $\,^1\chi_i=u_i$ and $\,^2\chi_i=0$. By Lemma~\ref{lem:cup product}, $0=\chi_1\cup\chi_2=\cdots=\chi_{n-1}\cup\chi_n$. Thus the $n$-fold Massey product $\langle\chi_1,\ldots,\chi_n\rangle$ is defined. Then there exists a homomorphism $\rho\colon G_1*G_2\to \bar{\U}_{n+1}(\F_p)$ such that $\rho_{i,i+1}=-\chi_i$ for $i=1,\ldots,n$. We have
\[
(\,^1\rho)_{i,i+1}=\,^1(\rho_{i,i+1})=-\,^1\chi_i=-u_i.
\]
This implies that $\langle u_1,\ldots,u_n\rangle$ is defined. Therefore $G_1$ has  the cup-defining $n$-fold Massey product property.

Similarly $G_2$ has  the cup-defining $n$-fold Massey product property.
\end{proof} 

In the following lemma we use standard results about skew-symmetric bilinear forms. It is worth to point out that the result below on non-alternate form is due to A. A. Albert \cite[Theorem 6]{A}.
\begin{lem}
\label{lem:Jacobson LA}
  Let $V$ be an $\F_p$-vector space of finite dimension $d\geq 3$. Let $(\cdot,\cdot)\colon V\times V\to \F_p$ be a  skew-symmetric bilinear from on $V$. Then there exists a basis $v_1,v_2,\ldots,v_n$ of $V$ such that 
\[
(v_1,v_2)=(v_2,v_3)=\cdots=(v_{d-1},v_d)=0.
\]
\end{lem}
\begin{proof} 
If the form $(\cdot,\cdot)$ is alternate, i.e., $(v,v)=0$ for all $v\in V$, then \cite[Chapter V, Theorem 7]{Ja} implies that there exists a basis $\{u_1,w_1,u_2,w_2,\ldots, u_r,w_r, z_1,\ldots,z_{d-2r}\}$ for $V$ such that
$V$ is the orthogonal sum $\sL(u_1,w_1)\perp \cdots\perp \sL(u_r,w_r)\perp \F_pz_1\perp\cdots \perp \F_p z_{d-2r} $, here $\sL(u_i,w_i)$ is the vector subspace  in $V$ generated by $u_i,w_i$. If $r=0$ then $z_1,...,z_d$ is our desired basis.
If $r=1$ then we  set $v_1=u_1$, $v_2=z_1,\ldots,v_{d-1}=z_{d-2}$, $v_d=w_1$, and  this is the desired basis.
If $r\geq 2$ then we set
\[
\begin{aligned}
v_i:=u_i & \text{ for } i=1,\ldots,r,\\
v_{r+i}:=w_i &\text{ for } i=1,\ldots,r,\\
v_{2r+i}=z_i &\text{ for } i=1,\ldots,d-2r.
\end{aligned}
\]
Then $\{v_1,\ldots,v_d\}$ is the desired basis

Assume now that the form $(\cdot,\cdot)$ is not alternate. Then \cite[Chapter V, Theorem 10]{Ja} implies that there exists a basis $v_1,\ldots,v_d$ for $V$ such that the matrix of $(\cdot,\cdot)$ relative to this basis is a diagonal matrix. In particular, we have
\[
(v_1,v_2)=(v_2,v_3)=\cdots=(v_{d-1},v_d)=0,
\]
as desired.
\end{proof}
A  finitely generated pro-$p$-group $G$ is a {\it one-relator} pro-$p$-group if   \[\dim_{\F_p} H^2(G,\F_p) =1.\] 
\begin{cor}
\label{cor:one-relator}
Let $G$ be a   finitely generated pro-$p$-group of rank $d\geq 3$. Assume that either $G$ is free or one-relator. Then there exists an $\F_p$-basis $\chi_1,\chi_2,\ldots,\chi_d$ for $H^1(G,\F_p)$ such that
\[
\chi_1\cup \chi_2=\chi_2\cup \chi_3=\cdots=\chi_{d-1}\cup \chi_d=0.
\]
\end{cor}
\begin{proof} This is clear if  $G$ is free. Assume that $G$ is one-relator. Then we can choose an isomorphism $\psi\colon H^2(G,\F_p)\simeq \F_p$ of $\F_p$-vector spaces. The bilinear form 
\[ 
H^1(G,\F_p)\times H^1(G,\F_p)\stackrel{\cup}{\to}H^2(G,\F_p)\stackrel{\psi}{\to}\F_p
\] is skew-symmetric. The statement then follows from Lemma~\ref{lem:Jacobson LA}.
\end{proof}

Assume now that $G=G_1*G_2$ is the free product of two finitely generated pro-$p$-groups $G_1$ and $G_2$. The restriction maps $res_i\colon H^1(G,\F_p)\to H^1(G_i,\F_p)$, $i=1,2$, define a homomorphism
\[
res\colon H^1(G,\F_p)\to  H^1(G_1,\F_p)\oplus H^1(G_2,\F_p).
\]
By \cite[Theorem 4.1.4]{NSW}, this homomorphism $res$ is an isomorphism. We shall identify $H^1(G,\F_p)$ with $H^1(G_1,\F_p)\oplus H^1(G_2,\F_p)$ via $res$.  Lemma~\ref{lem:cup product} implies that $u\cup v=0\in H^2(G,\F_p)$ for every $u\in H^1(G_1,\F_p)$, $v\in H^1(G_2,\F_p)$. We then immediately obtain the following lemma.

\begin{lem}
\label{lem:nice basis}
 Assume that $u_1,\ldots,u_d$ (respectively, $v_1,\ldots,v_e$) is an $\F_p$-basis of $H^1(G_1,\F_p)$ (respectively,  $H^1(G_2,\F_p)$) such that
\[
\begin{aligned}
&u_1\cup u_2=u_2\cup u_3=\cdots=u_{d-1}\cup u_d=0 \\
 &(respectively, v_1\cup v_2=v_2\cup v_3=\cdots=v_{e-1}\cup v_e=0).
 \end{aligned}
\]
Then $u_1,\ldots,u_d,v_1,\ldots,v_e$ is an $\F_p$-basis for $H^1(G,\F_p)$ which satisfies that
\[
0=u_1\cup u_2=\cdots=u_{d-1}\cup u_d=u_d\cup v_1=\cdots=v_{e-1}\cup v_{e}. \qedhere
\]
\end{lem}
\begin{thm}
\label{thm: Demushkin-Un}
Let $G$ be a finitely generated pro-$p$-group of rank $d$. Assume that $G={*}_{i\in I} G_i$ is a finite free product of pro-$p$-groups $G_i$ ($|I|=1$ is allowed), where each $G_i$ is either a free pro-$p$-group or a Demushkin group. Let $n$ be a natural number $\geq 2$. 
 Then $\U_n(\F_p)$ is a homomorphic image of $G$ if and only if $n\leq d+1$.
\end{thm}
\begin{proof}
Suppose first that there exists a surjective homomorphism $\rho\colon G\to \U_n(\F_p)$. Then the homomorphism 
\[
\varphi=(\rho_{12},\ldots,\rho_{n-1,n})\colon G\to \F_p\times \cdots \times \F_p
\]
induced by the projection of $\U_n(\F_p)$ on its near-by diagonal is also surjective. Hence by Lemma~\ref{lem:linear independence}  we see that $\rho_{12},\ldots,\rho_{n-1,n}$ are $\F_p$-linearly independent   in $H^1(G,\F_p)$. Therefore 
\[
n-1\leq \dim_{\F_p} H^1(G,\F_p)=d.
\]

We assume now that $n\leq d+1$. 

If $n=2$ then $\U_2(\F_p)\simeq \F_p$ and $1\leq d$. Hence $\U_2(\F_p)$ is a quotient of $G$. We suppose that $n\geq 3$. 
By Corollary~\ref{cor:one-relator} and Lemma~\ref{lem:nice basis} we can find a basis $v_1,\ldots,v_d$ for $H^1(G,\F_p)$ such that 
\[
v_1\cup v_2=v_2\cup v_3=\cdots = v_{d-1}\cup v_{d}=0.
\]
Then by Proposition~\ref{prop: Demushkin n-Massey} the $n-1$-fold Massey product $\langle v_1,\ldots,v_{n-1}\rangle$ is defined, and hence contains 0 by \cite[Theorem 4.2]{MT1}. Thus we obtain a representation 
$\rho\colon G\to \U_n(\F_p)$ such that $\rho_{i,i+1}=-v_i$ for every $i=1,2,\ldots,n-1$. By Lemma~\ref{lem:linear independence}, the induced homomorphism 
\[ \varphi=(\rho_{12},\ldots,\rho_{n-1,n})\colon G\to \F_p\times \cdots\times \F_p\]
is surjective. By usual Frattini argument applied to target group $\U_n(\F_p)$, this implies that $\rho$ is also surjective.
\end{proof}
\begin{rmk} In the famous Bourbaki report on Demushkin groups \cite{Se1}, J. -P. Serre in 1962 finally obtained the explicit description of $G_{\Q_2}(2)$ via generators and relations. I. R. Shafarevich in his very interesting article \cite[Section 2]{Sha2}, highlighted this group as the most curious group. We observe that we have now complete information about the number of Galois $\U_n(\F_2)$-extensions of $\Q_2$ for every $n\geq 2$. If $n=2$ then this number is 7.
If $n=3$ then this number is $18$ (see Remark~\ref{rmk:CP}). If $n=4$ then this number is 16 by Theorem~\ref{thm: local U4-extensions}. Finally from Theorem~\ref{thm: Demushkin-Un} we see that there are no Galois $\U_n(\F_2)$-extensions for $n\geq 5$. In \cite{AMT} we list all $\U_n(\F_2)$-Galois extensions of $\Q_2$ for all $n\geq 2$.
\end{rmk}
\section{Free products of Demushkin groups and free pro-$p$-groups}
Recall that we have the following exact sequence of finite groups
\[
1 \longrightarrow M \longrightarrow \U_4(\F_p) \xrightarrow{(a_{12},a_{23},a_{34})} (\F_p)^3\longrightarrow 1,
\]
where $a_{ij}\colon\U_4(\F_p)\to\F_p$ is the map sending a matrix to its $(i,j)$-coefficient, and $M$ is the kernel of the indicated homomorphism $(a_{12},a_{23},a_{34})$. This exact sequence induces a $B:=(\F_p)^3$-module structure on $M$.

Let $G=G_1*G_2$ be a free product in the category of pro-$p$-groups of two pro-$p$-groups $G_1$ and $G_2$.
Now let $\varphi\colon G\to B$ be any surjective continuous homomorphism.  Let $\varphi_1$ (respectively, $\varphi_2$) be the composition $\varphi_1\colon G_1\to G_1*G_2\stackrel{\varphi}{\to} B$ (respectively, $\varphi_2\colon G_2\to G_1*G_2\stackrel{\varphi}{\to} B$).
We consider $M$ as a $G$-module, denoted $M_\varphi$, via $\varphi$ and the action of $B$ on $M$. Note that $M$ is killed by $p$. Similarly, we have $G_1$-module $M_{\varphi_1}$ and $G_2$-module $M_{\varphi_2}$.

\subsection{A free product of a Demushkin group and a free pro-$p$-group}
In this subsection we assume that  $G_1$ is a Demushkin group of rank $d$, and that $G_2$ is a free pro-$p$-group of rank $e$. 

\begin{lem}
\label{lem:counting general Z1}
\mbox{}
\begin{enumerate}
\item If $\im \varphi_1\subseteq \{0\}\times \F_p\times \{0\}$ then  $|Z^1(G,M_{\varphi})|=p^{3d+3e}$.
\item If $\im \varphi_1\not\subseteq \{0\}\times \F_p\times \{0\}$ then $|Z^1(G,M_{\varphi})|=p^{3d+3e-1}$.
\end{enumerate}
\end{lem}
\begin{proof} Since $G_2$ is free of rank $e$, $\dim_{\F_p} Z^1(G_2,M_{\varphi_2})=3e$ by Example~\ref{ex:Z1 free}. Hence  
\[ \dim_{\F_p} Z^1(G,M_\varphi)=\dim_{\F_p}Z^1(G_1,M_{\varphi_1})+\dim_{\F_p}Z^1(G_2,M_{\varphi_2})=\dim_{\F_p}Z^1(G_1,M_{\varphi_1})+3e,\]
because $ Z^1(G,M_\varphi)\cong Z^1(G_1,M_{\varphi_1})\oplus Z^1(G_2,M_{\varphi_2})$ (see \cite[Proof of Theorem 4.1.4]{NSW}). The lemma then follows from Lemma~\ref{lem:counting Z1}.
\end{proof}
\begin{lem}
\label{lem:general LA}
Let $V$ and $W$ be two finite dimensional vector spaces over the finite field $\F_p$ with basis $\{v_1,v_2,\ldots,v_{d}\}$ and $\{w_1,\ldots,w_e\}$ respectively. Let $(\cdot,\cdot)\colon (V\oplus W)\times (V\oplus W)\to \F_p$ be a skew-symmetric bilinear form on $V\oplus W$. Assume that $(\cdot,\cdot)$ is non-degenerate when restricted to $V$ and trivial on $W$, and that $(v,w)=0$ for all $v\in V$, $w\in W$.
Let $N$ be the number of triples $(x,y,z)\in (V\oplus W)^3$ such that $(x,y)=(y,z)=0$ and that $x,y,z$ are $\F_p$-linearly independent.
\begin{enumerate}

\item If $(v_i,v_i)=0$ for every $1\leq i\leq d$, then 
\[N=(p^d-1)p^e(p^{d+e-1}-p)(p^{d+e-1}-p^2)+(p^e-1)(p^{d+e}-p)(p^{d+e}-p^2).\]
\item If $(v_1,v_1)=1$ and $(v_i,v_i)=0$ for every $2\leq i\leq d$, then $p=2$ and
\[
\begin{split}
N=(2^{d+e-1}-2)(2^{2d+2e-1}+3\cdot 2^{d+2e-1}-9\cdot2^{d+e-1}+4).
\end{split}
\]
\end{enumerate}
\end{lem}
\begin{proof}
 For each $y\in V\oplus W \setminus \{0\}$, we denote 
\[
 y^\perp =\{ x\in V\oplus W \mid (x,y)=0\}.
\]
Then $y^\perp$ is an $\F_p$-vector space, and 
\[
\dim y^\perp=
\begin{cases}
\dim (V\oplus W) = d+e \text{ if } y\in W,\\
\dim V+ \dim W-1 = d+e-1 \text{ if } y\not\in W.
\end{cases}
\]

We set 
\[
\begin{aligned}
 M(y)&:=\{(x,z)\in (V\oplus W)^2 \mid (x,y)=(y,z)=0; x,y,z \text{ are $\F_p$-linearly independent}\}\\
&=\{(x,z)\in y^\perp\times y^\perp\mid x,y,z \text{ are $\F_p$-linearly independent}\}.
\end{aligned}
\]

(1) It is similar to Proof of Lemma~\ref{lem:LA}, we  can check that $y\in y^\perp$. Indeed, writing $y=\sum_{i\in I} a_i v_i+w$, $a_i\in \F_p$, $I\subseteq \{1,2,\ldots,d\}$, $w\in W$, then 
\[
\begin{aligned}
 (y,y)&=(\sum_{i\in I} a_i v_i+w, \sum_{j\in I} a_j v_j+w)=\sum_{i,j\in I}a_ia_j(v_i,v_j)+ \sum_{i\in I} a_i(v_i,w)+\sum_{i\in I} a_i(w,v_i)+(w,w)\\
&=\sum_{i\in I} a_i^2(v_i,v_i)+\sum_{\substack{ i\not=j,\\ i,j\in I}}a_ia_j(v_i,v_j)=0,
\end{aligned}
\]
because $(\cdot,\cdot)$ is skew-symmetric and $(v_i,v_i)=0$ for $1\leq i\leq d$, $(w,w)=0$ by assumption. Thus 
\[
 |M(y)|=
 \begin{cases}
 (p^{d+e}-p)(p^{d+e}-p^2) \text{  if } y\in W,\\
  (p^{d+e-1}-p)(p^{d+e-1}-p^2) \text{  if } y\not\in W.
 \end{cases}
\]
Therefore
\[
 N=\sum_{y\in V\oplus W\setminus\{0\}} |M(y)|=(p^e-1)(p^{d+e}-p)(p^{d+e}-p^2)+p^e(p^d-1)(p^{d+e-1}-p)(p^{d+e-1}-p^2),
\]
as desired.
\\
\\
(2) Let us write $y=\sum_{i\in I}v_i+w$, $I\subseteq \{1,2,\ldots,d\}$, $w\in W$.

{\bf Case 1:} $1\not\in I$. Then we claim that $y\in y^\perp$. Indeed
\[
\begin{aligned}
 (y,y)&=(\sum_{i\in I} v_i+w, \sum_{j\in I} v_j+w)=\sum_{i,j\in I}(v_i,v_j)+\sum_{i\in I}(v_i,w)+\sum_{i\in I}(w,v_i)+(w,w)\\
&=\sum_{i\in I} (v_i,v_i)+\sum_{\substack{ i\not=j,\\ i,j\in I}}(v_i,v_j)=0,
\end{aligned}
\]
because $(\cdot,\cdot)$ is skew-symmetric, $(w,w)=0$, and $(v_i,v_i)=0$ for $2\leq i\leq d$ by assumption. Thus 
\[
 |M(y)|=
 \begin{cases}
 (2^{d+e}-2)(2^{d+e}-4) \text{ if } y\in W,\\
 (2^{d+e-1}-2)(2^{d+e-1}-4). \text{ if } y\not\in W
 \end{cases}
\]

{\bf Case 2:} $1\in I$. Then we claim that $(y,y)=1$, and hence $y\not\in y^\perp$. Indeed, let $I^\prime:=I\setminus\{1\}$, then
\[
\begin{aligned}
 (y,y)&=(v_1+\sum_{i\in I^\prime} v_i+w, v_1+\sum_{i\in I^\prime} v_j+w)\\
&=(v_1,v_1)+ (v_1,\sum_{i\in I^\prime} v_i)+(\sum_{i\in I^\prime} v_i,v_1)+(\sum_{i\in I^\prime} v_i,\sum_{i\in I^\prime} v_i)\\
&=(v_1,v_1)=1,
\end{aligned}
\]
because $(\cdot,\cdot)$ is skew-symmetric, $(w,w)=0$ and $(\sum_{i\in I^\prime} v_i,\sum_{i\in I^\prime} v_i)=0$ by Case 1. 
We note also that $y$ is linearly independently from $y^\perp$, 
this means if $ay + bz = 0$ for some $z$ in $y^\perp$, then $a=b=0$.
(In fact, $0=(ay+bz,y)=a(y,y)+b(z,y)=a$, and hence $b=0$ also.) 
Thus $|M(y)|$ is  the number of pairs $(x,z)\in y^{\perp}$ such that $x,z$ are $\F_2$-linearly independent. Therefore  
\[
 |M(y)|=(2^{d+e-1}-1)(2^{d+e-1}-2).
\]

We then obtain
\[
\begin{aligned}
 N=&\sum_\text{$y$ in Case 1} |M(y)|+ \sum_\text{$y$ in Case 2} |M(y)|\\
=&(2^e-1)(2^{d+e}-2)(2^{d+e}-4)+2^e(2^{d-1}-1)(2^{d+e-1}-2)(2^{d+e-1}-4) +\\
&2^{d+e-1}(2^{d+e-1}-1)(2^{d+e-1}-2)\\
=& (2^{d+e-1}-2)(2^{2d+2e-1}+3\cdot 2^{d+2e-1}-9\cdot2^{d+e-1}+4),
\end{aligned}
\]
as desired.
\end{proof}

\begin{prop} Let $G=G_1*G_2$ be the free product of a Demushkin group $G_1$ of rank $d$ and a free group of rank $e$.  
Let $N$ be the number as in Lemma~\ref{lem:general LA}. Then $|{\rm Epi}(G,\U_4(\F_p))|$ is equal to
\[
\begin{split}
Np^{3d+3e-1}+ [p^d(p^e-1)(p^e-p)(p^e-p^2)+(p^d-1)p^2(p^e-1)(p^e-p)](p^{3d+3e}-p^{3d+3e-1}).
\end{split}
\]
\end{prop}
\begin{proof}
Let $\varphi\in {\rm TMP}(G,\F_p)$. We also consider $\varphi$ as a surjective homomorphism $G\to B:=(\F_p)^3$.  Let $\varphi_1\colon G_1\to B$ be the composition $G_1\to G_1*G_2\stackrel{\varphi}{\to} B$ . We define $\varphi_2\colon G_2\to B$ similarly. Note that $\varphi$ is surjective if and only if $\im \varphi_1+\im\varphi_2=B$.
\\
\\
{\bf Case 1}: $\im\varphi_1=0\in B$. Then $\varphi$ is surjective if and only $\varphi_2$ is surjective. The number of $\varphi$'s in ${\rm TMP}(G,\F_p)$ with $\im\varphi_1=0$ is
\[
(p^e-1)(p^e-p)(p^e-p^2).
\]
{\bf Case 2}: $\im\varphi_1=\{0\}\times \F_p\times \{0\}$. 

{\bf Subcase 2.1}: $\im \varphi_2=B$. The number of such $\varphi$'s in ${\rm TMP}(G,\F_p)$  in this subcase is 
\[
(p^d-1)(p^e-1)(p^e-p)(p^e-p^2). 
\]

{\bf Subcase 2.2:} $\im \varphi_2\not=B$ and $\im\varphi_2+\{0\} \times \F_p \times \{0\}=B$. Then there exist $\lambda,\mu\in \F_p$ such that $\im\varphi_2=\{(a,\lambda a+\mu b,b)\mid a,b\in \F_p\}$. 
The number of such $\varphi$'s in ${\rm TMP}(G,\F_p)$  in this subcase is 
\[
(p^d-1)p^2(p^e-1)(p^e-p).
\]
{\bf Case 3}: $\im\varphi_1\not\subseteq \{0\}\times \F_p\times \{0\}$. The number of such $\varphi$'s in ${\rm TMP}(G,\F_p)$ in this case is just 
\[
N-M,
\]
where $N=|{\rm TMP}(G,\F_p)|$, which is computed as in Lemma~\ref{lem:general LA}, and $M$ is given by
\[
\begin{aligned}
M&:=&(p^e-1)(p^e-p)(p^e-p^2)+(p^d-1)(p^e-1)(p^e-p)(p^e-p^2)\\
&&+(p^d-1)p^2(p^e-1)(p^e-p)\\
&=& p^d(p^e-1)(p^e-p)(p^e-p^2)+(p^d-1)p^2(p^e-1)(p^e-p)
\end{aligned}.
\]
Combining with Proposition \ref{prop:counting Epi} and Lemma~\ref{lem:counting general Z1}, we obtain the result.
\end{proof}
Observe that setting $d = 0$ we recover a special case of Shafarevich's result in \cite{Sha} related to the case when $G$ is a free pro-$p$-group.
\subsection{A free product of two Demushkin groups}
In this subsection we assume that $G_1$ and $G_2$ are Demushkin groups of rank $d$ and $e$ respectively. For simplicity we shall consider the case both $q$-invariants of $G_1$ and $G_2$ are greater than 2.

\begin{lem}
\label{lem:counting general Z1 DD}
\mbox{}
\begin{enumerate}
\item If either $\im \varphi_1$ or $\im\varphi_2$, but not both, is a subspace of $\{0\}\times \F_p\times \{0\}$ then  $|Z^1(G,M_{\varphi})|=p^{3d+3e-1}$.
\item If neither $\im \varphi_1$ nor $\im\varphi_2$ is  a subspace of $\{0\}\times \F_p\times \{0\}$ then $|Z^1(G,M_{\varphi})|=p^{3d+3e-2}$.
\end{enumerate}
\end{lem}
\begin{proof} Note that  
\[ \dim_{\F_p} Z^1(G,M_\varphi)=\dim_{\F_p}Z^1(G_1,M_{\varphi_1})+\dim_{\F_p}Z^1(G_2,M_{\varphi_2}),\]
because $ Z^1(G,M_\varphi)\cong Z^1(G_1,M_{\varphi_1})\oplus Z^1(G_2,M_{\varphi_2})$ (see \cite[Proof of Theorem 4.1.4]{NSW}). The lemma then follows from Lemma~\ref{lem:counting Z1}.
\end{proof}
\begin{lem}
\label{lem:general LA DD}
Let $V$ and $W$ be two finite dimensional vector spaces over the finite field $\F_p$ with basis $\{v_1,v_2,\ldots,v_{d}\}$ and $\{w_1,\ldots,w_e\}$ respectively. Let $(\cdot,\cdot)_1, (\cdot,\cdot)_2 \colon (V\oplus W)\times (V\oplus W)\to \F_p$ be two skew-symmetric bilinear forms on $V\oplus W$. Assume that 
\begin{enumerate}
\item[(a)] $(\cdot,\cdot)_1$ is non-degenerate when restricted to $V$ and trivial on $W$; and that
\item[(b)] $(\cdot,\cdot)_2$ is non-degenerate when restricted to $W$ and trivial on $V$; and that
\item[(c)] $(v,w)_1=(v,w)_2=0$ for all $v\in V$, $w\in W$.
\end{enumerate}
Let $N$ be the number of triples $(x,y,z)\in (V\oplus W)^3$ such that $(x,y)_1=(x,y)_2=(y,z)_1=(y,z)_2=0$ and that $x,y,z$ are $\F_p$-linearly independent. The following statement is true.

 If $(v_i,v_i)=0$ for every $1\leq i\leq d$ and $(w_j,w_j)=0$ for every $1\leq j\leq e$, then 
\[N=(p^d+p^e-2)(p^{d+e-1}-p)(p^{d+e-1}-p^2)+(p^d-1)(p^e-1)(p^{d+e-2}-p)(p^{d+e-2}-p^2).\]
\end{lem}
\begin{proof}
 For each $y\in V\oplus W \setminus \{0\}$, we denote 
\[
 y^\perp =\{ x\in V\oplus W \mid (x,y)_1=(x,y)_2=0\}.
\]
Then $y^\perp$ is an $\F_p$-vector space, and 
\[
\dim y^\perp=
\begin{cases}
\dim V+\dim W-1 = d+e-1 \text{ if } y\in V\text{ or } y\in W,\\
\dim V-1 + \dim W-1 = d+e-2 \text{ if } y\not\in V\cup W.
\end{cases}
\]

We set 
\[
\begin{aligned}
 M(y)&:=\begin{split}&\{(x,z)\in (V\oplus W)^2 \mid (x,y)_1=(x,y)_2=(y,z)_1=(y,z)_2=0; \\
                     &x,y,z  \text{ are $\F_p$-linearly independent}\}
\end{split}\\
&=\{(x,z)\in y^\perp\times y^\perp\mid x,y,z \text{ are $\F_p$-linearly independent}\}.
\end{aligned}
\]

We  can check that $y\in y^\perp$ as in Proof of Lemma~\ref{lem:general LA}.  Thus 
\[
 |M(y)|=
 \begin{cases}
 (p^{d+e-1}-p)(p^{d+e-1}-p^2) \text{  if } y\in V\cup W,\\
  (p^{d+e-2}-p)(p^{d+e-2}-p^2) \text{  if } y\not\in V\cup W.
 \end{cases}
\]
Therefore
\[
\begin{split}
N&=\sum_{y\in V\oplus W\setminus\{0\}} |M(y)|\\
&=(p^d+p^e-2)(p^{d+e-1}-p)(p^{d+e-1}-p^2)+(p^d-1)(p^e-1)(p^{d+e-2}-p)(p^{d+e-2}-p^2),
\end{split}
\]
as desired.
\end{proof}

\begin{prop} Let $G=G_1*G_2$ be the free product of a Demushkin group $G_1$ of rank $d$ and a Demushkin group $G_2$ of rank $e$.  Assume that the $q$-invariants of both $G_1$ and $G_2$ are greater than 2.
Let $N$ be the number as in Lemma~\ref{lem:general LA DD}. Then $|{\rm Epi}(G,\U_4(\F_p))|$ is equal to
\[
\begin{split}
Np^{3d+3e-2}+ [& p^d(p^e-1)(p^e-p)(p^e-p^2)+ p^e(p^d-1)(p^d-p)(p^d-p^2)+ \\
                             &p^{2}(p^d-1)(p^e-1)(p^e-p)+p^{2}(p^e-1)(p^d-1)(p^d-p)](p^{3d+3e-1}-p^{3d+3e-2}).
\end{split}
\]
\end{prop}
\begin{proof}
Let $\varphi\in {\rm TMP}(G,\F_p)$. We also consider $\varphi$ as a surjective homomorphism $G\to B$. Let $\varphi_1\colon G_1\to B$ be the composition $G_1\to G_1*G_2\stackrel{\varphi}{\to} B$ . We define $\varphi_2\colon G_2\to B$ similarly. Note that $\varphi$ is surjective if and only if $\im \varphi_1+\im\varphi_2=B$.
\\
\\
{\bf Case 1}: $\im\varphi_1=0\in B$. Then $\varphi$ is surjective if and only $\varphi_2$ is surjective. The number of $\varphi$'s in ${\rm TMP}(G,\F_p)$ with $\im\varphi_1=0$ is
\[
(p^e-1)(p^e-p)(p^e-p^2).
\]
\\
\\
{\bf Case 2}: $\im\varphi_2=0\in B$.  The number of $\varphi$'s in ${\rm TMP}(G,\F_p)$ with $\im\varphi_2=0$ is
\[
(p^d-1)(p^d-p)(p^d-p^2).
\]
\\
\\
{\bf Case 3}: $\im\varphi_1=\{0\}\times \F_p\times \{0\}$. 

{\bf Subcase 3.1}: $\im \varphi_2=B$. The number of such $\varphi$'s in ${\rm TMP}(G,\F_p)$  in this subcase is 
\[
(p^d-1)(p^e-1)(p^e-p)(p^e-p^2). 
\]

{\bf Subcase 3.2:} $\im \varphi_2\not=B$ and $\im\varphi_2+\{0\} \times \F_p \times \{0\}=B$. Then there exist $\lambda,\mu\in \F_p$ such that $\im\varphi_2=\{(a,\lambda a+\mu b,b)\mid a,b\in \F_p\}$. 
The number of such $\varphi$'s in ${\rm TMP}(G,\F_p)$  in this subcase is 
\[
(p^d-1)p^2(p^e-1)(p^e-p).
\]
\\
\\
{\bf Case 4}: $\im\varphi_2=\{0\}\times \F_p\times \{0\}$. 

{\bf Subcase 3.1}: $\im \varphi_1=B$. The number of such $\varphi$'s in ${\rm TMP}(G,\F_p)$  in this subcase is 
\[
(p^e-1)(p^d-1)(p^d-p)(p^d-p^2). 
\]

{\bf Subcase 3.2:} $\im \varphi_1\not=B$ and $\im\varphi_1+\{0\} \times \F_p \times \{0\}=B$. Then there exist $\lambda,\mu\in \F_p$ such that $\im\varphi_1=\{(a,\lambda a+\mu b,b)\mid a,b\in \F_p\}$. 
The number of such $\varphi$'s in ${\rm TMP}(G,\F_p)$  in this subcase is 
\[
(p^e-1)p^2(p^d-1)(p^d-p).
\]
\\
\\
{\bf Case 4}: $\im\varphi_1$ and $\im\varphi_2$ are not contained in $\{0\}\times \F_p\times \{0\}$. The number of such $\varphi$'s in ${\rm TMP}(G,\F_p)$ in this case is just 
\[
N-(\text{ the sum of those numbers found in previous cases}),
\]
where $N=|{\rm TMP}(G,\F_p)|$, which is computed as in Lemma~\ref{lem:general LA DD}.

Combining with Proposition \ref{prop:counting Epi} and Lemma~\ref{lem:counting general Z1 DD}, we obtain the result.
\end{proof}

\end{document}